\documentclass[12pt,dvips]{amsart}
\usepackage{amsfonts, amssymb, latexsym, epsfig}

\newtheorem{Theorem}{Theorem}[section]
\newtheorem{Proposition}[Theorem]{Proposition}
\newtheorem{Lemma}[Theorem]{Lemma}
\newtheorem{Claim}[Theorem]{Claim}

\newtheorem{Corollary}[Theorem]{Corollary}

\newtheorem{Main Conjecture}[Theorem]{Main Conjecture}

\newtheorem{Definition}[Theorem]{Definition}

\theoremstyle{remark}
\newtheorem{Example}[Theorem]{Example}

\newcommand\olambda{{\overline\lambda}}
\newcommand\omu{{\overline\mu}}

\newcommand\oalpha{{\overline\alpha}}

\newcommand\okappa{{\overline\kappa}}

\newcommand{\excise}[1]{}

\theoremstyle{plain}

\begin{document}
\pagestyle{plain}
\title{Root-theoretic Young diagrams and Schubert calculus II}
\author{Dominic Searles}
\address{Department of Mathematics\\
University of Illinois at Urbana-Champaign\\
Urbana, IL 61801}
\email{searles2@uiuc.edu}

\date{November 12, 2013.}

\begin{abstract}
We continue the study of root-theoretic Young diagrams (RYDs) from \cite{Searles.Yong}. We provide an RYD formula for the $GL_n$ Belkale-Kumar product, after \cite{Knutson.Purbhoo}, and we give a translation of the indexing set of \cite{BKT:Inventiones} for Schubert varieties of non-maximal isotropic Grassmannians into RYDs. We then use this translation to prove that the RYD formulas of \cite{Searles.Yong} for Schubert calculus of the classical (co)adjoint varieties agree with the Pieri rules of \cite{BKT:Inventiones}, which were needed in the proofs of the (co)adjoint formulas.
\end{abstract}

\maketitle

\section{Introduction}

\subsection{Overview}

In \cite{Searles.Yong}, A. Yong and the author study {\it root-theoretic Young diagrams} (RYDs), which are one of several natural choices of indexing set for the Schubert subvarieties of generalized flag varieties. The thesis of that paper and the present one is that RYDs are useful for studying general patterns in Schubert combinatorics in a uniform manner. The main evidence introduced in \cite{Searles.Yong} is rules for Schubert calculus of the classical (co)adjoint varieties in terms of RYDs, and a relation between planarity of the root poset for a (co)adjoint variety and polytopalness of the nonzero Schubert structure constants for its cohomology ring. 

The problem of finding a nonnegative, integral combinatorial rule for the Schubert structure constants of the cohomology ring of a generalized flag variety is longstanding. Much progress has been made on this problem, see, e.g., the survey \cite{Coskun.Vakil}. One of the more recent areas of progress is in the study of the {\it Belkale-Kumar product}, introduced by P.~Belkale and S.~Kumar in \cite{Belkale.Kumar}. The structure constants of the Belkale-Kumar product in the case of $GL_n$ are described by a beautiful formula of A.~Knutson-K.~Purbhoo \cite{Knutson.Purbhoo} in terms of puzzles. In this paper, we use a factorization formula of \cite{Knutson.Purbhoo} to derive a new formula in terms of RYDs for the Belkale-Kumar product. 

We find that the RYD formula manifests in a simple way the product/factorization structure of the Belkale-Kumar coefficients in terms of Schubert structure constants of Grassmannians. In particular, RYDs allow us to visually reduce computation of these coefficients to a collection of independent calculations using the {\it jeu de taquin} algorithm of M.-P.~Sch\"utzenberger \cite{Schutzenberger}. The RYD description also provides a concrete context to explain in what sense the Belkale-Kumar product is ``easier'' than the cup product. Specifically, the RYDs naturally consist of a number of regions. In the rule for the Belkale-Kumar coefficients there is no interaction between these regions and they can be treated independently of each other. This is not true for the Schubert structure constants, e.g., Example~\ref{ex:Schubertregions} exhibits concretely how the Belkale-Kumar case differs from the general problem. 

We would like to study, compare and understand disparate models and problems in Schubert calculus through the common lens of RYDs. Towards this end, we consider also the family of non-maximal {\it isotropic Grassmannians}. 
A.~Buch-A.~Kresch-H.~Tamvakis \cite{BKT:Inventiones} define an indexing set for the Schubert varieties of non-maximal isotropic Grassmannians, and use this indexing set to give particularly nice Pieri rules for the Schubert calculus of these spaces.
The Schubert calculus formulas of \cite{Searles.Yong} for the (co)adjoint varieties of classical Lie type were discovered using the RYD model to index Schubert varieties. The proof of these formulas we requires Pieri rules for these (co)adjoint varieties, the most interesting of which belong to the family of non-maximal isotropic Grassmannians.
Therefore, we provide a reformulation of the indexing set of \cite{BKT:Inventiones} in terms of RYDs.

For the classical (co)adjoint varieties, we use this reformulation to prove that the restriction to the Pieri cases of the formulas of \cite[Theorem 4.1]{Searles.Yong} and \cite[Theorem 5.3]{Searles.Yong} agrees with the Pieri rule of \cite{BKT:Inventiones}. 
In tandem with the proofs of associativity of these (co)adjoint formulas given in \cite{Searles.Yong}, this completes the proofs of these (co)adjoint formulas.
\subsection{The Belkale-Kumar product for $GL_n/P$}

The Belkale-Kumar product is a certain deformation of the usual cup product for $H^\star(G/P)$. Our first result is an RYD formula for this product in the case where $G=GL_n$, after \cite{Knutson.Purbhoo}. RYDs are in fact defined for any generalized flag variety $G/P$, where $G$ is a complex reductive Lie group and $P$ is a parabolic subgroup of $G$; see \cite{Searles.Yong} for further details. In this section, for brevity, we set $G=GL_n$. 

Fix a set ${\tt k}=\{k_1,\ldots,k_{d-1}\}$ of integers satisfying $0<k_1<\ldots <k_{d-1}<n$. 
Let $F_{{\tt k}}:=Fl_{k_1,\ldots , k_{d-1};\mathbb{C}^n}$ denote the $(d-1)$-step {\bf flag variety} in $\mathbb{C}^n$, where the $d-1$ nested subspaces of $\mathbb{C}^n$ have dimensions $k_1,\ldots,k_{d-1}$.
The {\bf Schubert varieties} of $F_{{\tt k}}$ are indexed by the set $S_n^{{\tt k}}$ which consists of the elements of the symmetric group $S_n$ that have descents only in positions $k_1,\ldots,k_{d-1}$. In the case of $F_{{\tt k}}$, the RYDs of \cite{Searles.Yong} are the inversion sets of the elements of $S_n^{{\tt k}}$ in the poset $\Omega_{GL_n}$ of positive roots of $GL_n$. Let $\mathbb{Y}_{{\tt k}}$ denote the set of RYDs for $F_{{\tt k}}$.

Let $I_i$ denote the interval $[k_{i-1}+1,k_i]$ for $1\le i\le d$, where we set $k_0=0$ and $k_d=n$. Let $(a,b)\in \Omega_{GL_n}$ index the root $e_a - e_b$ under the standard embedding of the type $A_{n-1}$ root system into $\mathbb{R}^n$. For each pair $i,j$ with $1\le i < j\le d$, we define an associated {\bf region} $\Lambda_{{\tt k}}^{ij}:=I_i\times I_j$ of $\Omega_{GL_n}$. We will show in the following section (Claim~\ref{claim:permtodiagram}) that each RYD $\lambda\in\mathbb{Y}_{{\tt k}}$ consists of a lower order ideal in each of these $d\choose 2$ regions. 
		
\begin{Example}\label{ex:RYD}
Let $n=7$ and ${\tt k} = \{1,3,5\}$. Then $F_{{\tt k}}=Fl_{1,3,5;\mathbb{C}^7}$, and we have (in one-line notation) $5371624$, $3462715\in S_7^{{\tt k}}$. Below, their RYDs are shown as a subset (colored black) of the poset $\Omega_{GL_7}$. The thicker black lines show the regions $\Lambda_{{\tt k}}^{ij}$.

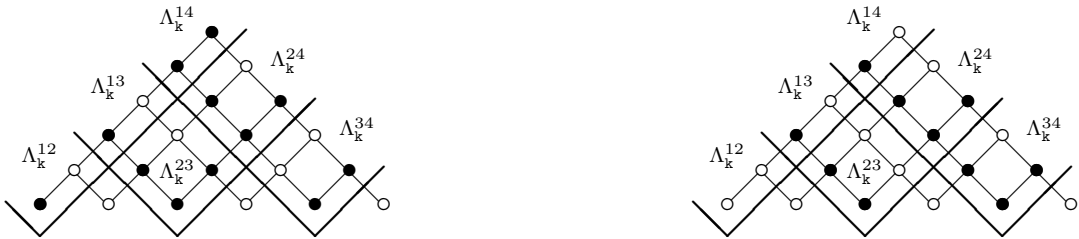
\[\begin{picture}(400,80)
\multiput(4.5,17.5)(13,13){5}{\line(1,1){10}}
\multiput(0,13)(13,13){6}{$\circ$}
\multiput(26,13)(13,13){5}{$\circ$}
\multiput(52,13)(13,13){4}{$\bullet$}
\multiput(78,13)(13,13){3}{$\circ$}
\multiput(104,13)(13,13){2}{$\bullet$}
\multiput(130,13)(13,13){1}{$\circ$}
\multiput(30.5,17.5)(13,13){4}{\line(1,1){10}}
\multiput(56.5,17.5)(13,13){3}{\line(1,1){10}}
\multiput(82.5,17.5)(13,13){2}{\line(1,1){10}}
\multiput(108.5,17.5)(13,13){1}{\line(1,1){10}}
\multiput(17,27)(13,13){5}{\line(1,-1){10}}
\multiput(43,27)(13,13){4}{\line(1,-1){10}}
\multiput(69,27)(13,13){3}{\line(1,-1){10}}
\multiput(95,27)(13,13){2}{\line(1,-1){10}}
\multiput(121,27)(13,13){1}{\line(1,-1){10}}

\put(0,13){$\bullet$}
\put(26,39){$\bullet$}
\multiput(52,65)(13,13){2}{$\bullet$}
\multiput(39,26)(26,26){2}{$\bullet$}

\thicklines
\put(3,4){\line(-1,1){13}}
\put(3,4){\line(1,1){78}}

\put(55,4){\line(-1,1){39}}
\put(55,4){\line(1,1){52}}

\put(107,4){\line(-1,1){65}}
\put(107,4){\line(1,1){26}}

\thinlines

\put(-4,32){{\tiny $\Lambda^{12}_{{\tt k}}$}}
\put(22,58){{\tiny $\Lambda^{13}_{{\tt k}}$}}
\put(48,84){{\tiny $\Lambda^{14}_{{\tt k}}$}}
\put(90,68){{\tiny $\Lambda^{24}_{{\tt k}}$}}
\put(116,42){{\tiny $\Lambda^{34}_{{\tt k}}$}}
\put(48,26){{\tiny $\Lambda^{23}_{{\tt k}}$}}


\multiput(264.5,17.5)(13,13){5}{\line(1,1){10}}
\multiput(260,13)(13,13){6}{$\circ$}
\multiput(286,13)(13,13){5}{$\circ$}
\multiput(312,13)(13,13){4}{$\circ$}
\multiput(338,13)(13,13){3}{$\circ$}
\multiput(364,13)(13,13){2}{$\circ$}
\multiput(390,13)(13,13){1}{$\circ$}
\multiput(290.5,17.5)(13,13){4}{\line(1,1){10}}
\multiput(316.5,17.5)(13,13){3}{\line(1,1){10}}
\multiput(342.5,17.5)(13,13){2}{\line(1,1){10}}
\multiput(368.5,17.5)(13,13){1}{\line(1,1){10}}
\multiput(277,27)(13,13){5}{\line(1,-1){10}}
\multiput(303,27)(13,13){4}{\line(1,-1){10}}
\multiput(329,27)(13,13){3}{\line(1,-1){10}}
\multiput(355,27)(13,13){2}{\line(1,-1){10}}
\multiput(381,27)(13,13){1}{\line(1,-1){10}}

\multiput(286,39)(26,26){2}{$\bullet$}
\multiput(299,26)(26,26){2}{$\bullet$}
\multiput(312,13)(26,26){2}{$\bullet$}
\put(351,52){$\bullet$}
\put(351,26){$\bullet$}
\multiput(364,13)(13,13){2}{$\bullet$}

\thicklines
\put(263,4){\line(-1,1){13}}
\put(263,4){\line(1,1){78}}

\put(315,4){\line(-1,1){39}}
\put(315,4){\line(1,1){52}}

\put(367,4){\line(-1,1){65}}
\put(367,4){\line(1,1){26}}

\thinlines

\put(256,32){{\tiny $\Lambda^{12}_{{\tt k}}$}}
\put(282,58){{\tiny $\Lambda^{13}_{{\tt k}}$}}
\put(308,84){{\tiny $\Lambda^{14}_{{\tt k}}$}}
\put(350,68){{\tiny $\Lambda^{24}_{{\tt k}}$}}
\put(376,42){{\tiny $\Lambda^{34}_{{\tt k}}$}}
\put(308,26){{\tiny $\Lambda^{23}_{{\tt k}}$}}

\end{picture}
\]
\end{Example}

Let $C_{\lambda,\mu}^{\nu}(F_{{\tt k}})$ denote the {\bf Schubert structure constants} for the cohomology ring $H^\star(F_{{\tt k}})$, i.e.,
\[\sigma_{\lambda}\cdot\sigma_{\mu}=\sum_{\nu} C_{\lambda,\mu}^{\nu}(F_{{\tt k}})\sigma_{\nu}.\]

Let $\lambda_{ij}$ denote the restriction of $\lambda$ to the region $\Lambda_{{\tt k}}^{ij}$. Define a triple $(\lambda,\mu,\nu)\in (\mathbb{Y}_{{\tt k}})^3$ to be {\bf Levi-movable} if $C_{\lambda,\mu}^{\nu}(F_{{\tt k}})\neq 0$ and $|\lambda_{ij}|+|\mu_{ij}|=|\nu_{ij}|$ for all regions $\Lambda_{{\tt k}}^{ij}$. This is essentially identical to the inversion set definition of Levi-movability in the $GL_n$ case from \cite{Knutson.Purbhoo}. It follows from Theorem~\ref{Thm:BKRYD} below that, for $GL_n$, our definition is equivalent to the geometric definition of Levi-movability of \cite{Belkale.Kumar}. Define
\[b_{\lambda,\mu}^{\nu}(F_{{\tt k}})=\begin{cases}
C_{\lambda,\mu}^{\nu}(F_{{\tt k}}) & \text{if $(\lambda,\mu,\nu)$ is Levi-movable}\\
0 & \text{otherwise.} 
\end{cases}\]

Then the {\bf Belkale-Kumar product} $\odot_0$ on $H^\star(F_{{\tt k}})$ is defined by
\[\sigma_{\lambda}\odot_0\sigma_{\mu}=\sum_{\nu} b_{\lambda,\mu}^{\nu}(F_{{\tt k}})\sigma_{\nu}.\]

For further details, see \cite{Belkale.Kumar}. We also learned much of the background from \cite{Richmond}.

Our formula uses the {\bf jeu de taquin} introduced in \cite{Schutzenberger}. The following setup in terms of root posets is similar to that employed in \cite{Thomas.Yong:comin}. Given a subset $S$ of $\Lambda_{{\tt k}}^{ij}$, define a partial labelling $T_S$ of $\Lambda_{{\tt k}}^{ij}$ by bijectively assigning each root in $S$ a number from $\{1,\ldots,|S|\}$, subject to the condition that a root $\alpha$ receives a smaller number than a root $\alpha'$ whenever $\alpha\prec\alpha'$. Roots in $\Lambda_{{\tt k}}^{ij}$ that have no label will be called unlabelled. Let $\lambda,\mu,\nu\in \mathbb{Y}_{{\tt k}}$. Let $\nu/\lambda$ denote the set-theoretic difference of $\nu$ and $\lambda$, and call $\nu/\lambda$ a {\bf skew} RYD.  

Starting with a given labelling $T_{\nu_{ij}/\lambda_{ij}}$, choose an unlabelled root $\alpha$ of $\Lambda_{{\tt k}}^{ij}$ which is maximal subject to the condition that some labelled root is above it. Among the labelled roots covering $\alpha$, choose the root $\alpha'$ having the smallest label. Move its label to $\alpha$, leaving $\alpha'$ unlabelled. Then find the labelled root covering $\alpha'$ with smallest label, and move its label to $\alpha'$. Continue in this manner until a label is moved from a root that has no labelled root above it. Then, choose an unlabelled root of $\Lambda_{{\tt k}}^{ij}$, maximal such that some labelled root is above it and perform the same process. Repeat until there is no unlabelled root below a labelled root. Let ${\tt jdt}(T_{\nu_{ij}/\lambda_{ij}})$ denote the resulting partial labelling of $\Lambda_{{\tt k}}^{ij}$. 

Fix a choice of labelling $T_{\mu_{ij}}$. Let $e_{\lambda_{ij},\mu_{ij}}^{\nu_{ij}}$ denote the number of labellings $T_{\nu_{ij}/\lambda_{ij}}$ such that ${\tt jdt}(T_{\nu_{ij}/\lambda_{ij}})=T_{\mu_{ij}}$. 
Then the Belkale-Kumar coefficient $b_{\lambda,\mu}^\nu(F_{{\tt k}})$ is computed by taking the skew RYD $\nu/\lambda$, performing the jeu de taquin algorithm independently on each region of $\Omega_{GL_n}$, and multiplying the resulting numbers $e_{\lambda_{ij},\mu_{ij}}^{\nu_{ij}}$. In other words:

\begin{Theorem}\label{Thm:BKRYD}
\[b_{\lambda,\mu}^{\nu}(F_{{\tt k}})=\prod_{{\rm regions \ } \Lambda_{{\tt k}}^{ij}}e_{\lambda_{ij},\mu_{ij}}^{\nu_{ij}}.\]
\end{Theorem}

\begin{Example}
Let $n=7$ and ${\tt k}=\{3,6\}$. Then $F_{{\tt k}}=Fl_{3,6; \mathbb{C}^7}$, and $1362475$, $1462573$, $3572461\in S_7^{{\tt k}}$. Let (respectively) $\lambda,\mu,\nu$ be the corresponding RYDs. Below is a choice of labellings $\{T_{\mu_{ij}}\}$ of the RYD $\mu$, and the two labellings $\{T_{\nu_{ij}/\lambda_{ij}}\}$ of the skew RYD $\nu/\lambda$ such that ${\tt jdt}(T_{\nu_{ij}/\lambda_{ij}})=T_{\mu_{ij}}$ in each region $\Lambda_{{\tt k}}^{ij}$. 

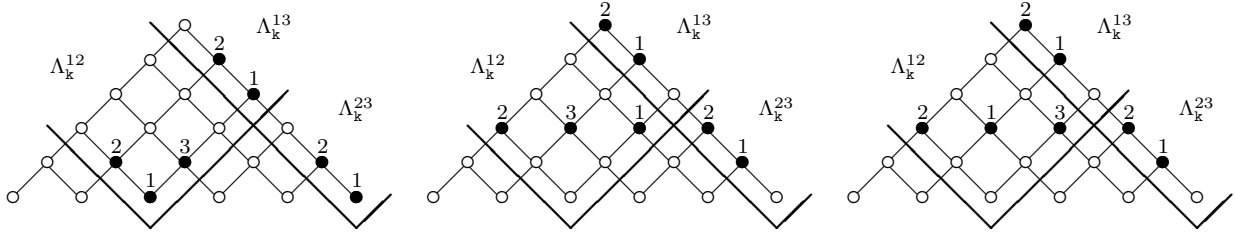
\[\begin{picture}(500,100)

\multiput(0,13)(13,13){6}{$\circ$}
\multiput(26,13)(13,13){5}{$\circ$}
\multiput(52,13)(13,13){4}{$\circ$}
\multiput(78,13)(13,13){3}{$\circ$}
\multiput(104,13)(13,13){2}{$\circ$}
\multiput(130,13)(13,13){1}{$\circ$}
\multiput(4.5,17.5)(13,13){5}{\line(1,1){10}}
\multiput(30.5,17.5)(13,13){4}{\line(1,1){10}}
\multiput(56.5,17.5)(13,13){3}{\line(1,1){10}}
\multiput(82.5,17.5)(13,13){2}{\line(1,1){10}}
\multiput(108.5,17.5)(13,13){1}{\line(1,1){10}}
\multiput(17,27)(13,13){5}{\line(1,-1){10}}
\multiput(43,27)(13,13){4}{\line(1,-1){10}}
\multiput(69,27)(13,13){3}{\line(1,-1){10}}
\multiput(95,27)(13,13){2}{\line(1,-1){10}}
\multiput(121,27)(13,13){1}{\line(1,-1){10}}

\put(39,26){$\bullet$}
\put(52,13){$\bullet$}
\put(65,26){$\bullet$}
\put(91,52){$\bullet$}
\put(130,13){$\bullet$}
\put(117,26){$\bullet$}
\put(78,65){$\bullet$}

\put(40,32.5){{\tiny $2$}}
\put(53,19.5){{\tiny $1$}}
\put(66,32.5){{\tiny $3$}}
\put(92,58.5){{\tiny $1$}}
\put(131,19.5){{\tiny $1$}}
\put(118,32.5){{\tiny $2$}}
\put(79,71.5){{\tiny $2$}}

\thicklines

\put(55,4){\line(-1,1){39}}
\put(55,4){\line(1,1){52}}

\put(133,4){\line(-1,1){78}}
\put(133,4){\line(1,1){13}}

\thinlines

\put(17,63){{\tiny $\Lambda^{12}_{{\tt k}}$}}
\put(95,78){{\tiny $\Lambda^{13}_{{\tt k}}$}}
\put(126,47){{\tiny $\Lambda^{23}_{{\tt k}}$}}


\multiput(159,13)(13,13){6}{$\circ$}
\multiput(185,13)(13,13){5}{$\circ$}
\multiput(211,13)(13,13){4}{$\circ$}
\multiput(237,13)(13,13){3}{$\circ$}
\multiput(263,13)(13,13){2}{$\circ$}
\multiput(289,13)(13,13){1}{$\circ$}

\multiput(163.5,17.5)(13,13){5}{\line(1,1){10}}
\multiput(189.5,17.5)(13,13){4}{\line(1,1){10}}
\multiput(215.5,17.5)(13,13){3}{\line(1,1){10}}
\multiput(241.5,17.5)(13,13){2}{\line(1,1){10}}
\multiput(267.5,17.5)(13,13){1}{\line(1,1){10}}
\multiput(176,27)(13,13){5}{\line(1,-1){10}}
\multiput(202,27)(13,13){4}{\line(1,-1){10}}
\multiput(228,27)(13,13){3}{\line(1,-1){10}}
\multiput(254,27)(13,13){2}{\line(1,-1){10}}
\multiput(280,27)(13,13){1}{\line(1,-1){10}}

\thicklines

\put(214,4){\line(-1,1){39}}
\put(214,4){\line(1,1){52}}

\put(292,4){\line(-1,1){78}}
\put(292,4){\line(1,1){13}}

\thinlines

\put(176,63){{\tiny $\Lambda^{12}_{{\tt k}}$}}
\put(254,78){{\tiny $\Lambda^{13}_{{\tt k}}$}}
\put(285,47){{\tiny $\Lambda^{23}_{{\tt k}}$}}

\multiput(185,39)(26,0){4}{$\bullet$}
\put(224,78){$\bullet$}
\put(237,65){$\bullet$}
\put(276,26){$\bullet$}

\put(186,45.5){{\tiny $2$}}
\put(212,45.5){{\tiny $3$}}
\put(238,45.5){{\tiny $1$}}
\put(264,45.5){{\tiny $2$}}

\put(225,84.5){{\tiny $2$}}
\put(238,71.5){{\tiny $1$}}
\put(277,32.5){{\tiny $1$}}


\multiput(318,13)(13,13){6}{$\circ$}
\multiput(344,13)(13,13){5}{$\circ$}
\multiput(370,13)(13,13){4}{$\circ$}
\multiput(396,13)(13,13){3}{$\circ$}
\multiput(422,13)(13,13){2}{$\circ$}
\multiput(448,13)(13,13){1}{$\circ$}

\multiput(322.5,17.5)(13,13){5}{\line(1,1){10}}
\multiput(348.5,17.5)(13,13){4}{\line(1,1){10}}
\multiput(374.5,17.5)(13,13){3}{\line(1,1){10}}
\multiput(400.5,17.5)(13,13){2}{\line(1,1){10}}
\multiput(426.5,17.5)(13,13){1}{\line(1,1){10}}
\multiput(335,27)(13,13){5}{\line(1,-1){10}}
\multiput(361,27)(13,13){4}{\line(1,-1){10}}
\multiput(387,27)(13,13){3}{\line(1,-1){10}}
\multiput(413,27)(13,13){2}{\line(1,-1){10}}
\multiput(439,27)(13,13){1}{\line(1,-1){10}}

\thicklines

\put(373,4){\line(-1,1){39}}
\put(373,4){\line(1,1){52}}

\put(451,4){\line(-1,1){78}}
\put(451,4){\line(1,1){13}}

\thinlines

\put(335,63){{\tiny $\Lambda^{12}_{{\tt k}}$}}
\put(413,78){{\tiny $\Lambda^{13}_{{\tt k}}$}}
\put(444,47){{\tiny $\Lambda^{23}_{{\tt k}}$}}

\multiput(344,39)(26,0){4}{$\bullet$}
\put(383,78){$\bullet$}
\put(396,65){$\bullet$}
\put(435,26){$\bullet$}

\put(345,45.5){{\tiny $2$}}
\put(371,45.5){{\tiny $1$}}
\put(397,45.5){{\tiny $3$}}
\put(423,45.5){{\tiny $2$}}

\put(384,84.5){{\tiny $2$}}
\put(397,71.5){{\tiny $1$}}
\put(436,32.5){{\tiny $1$}}

\end{picture}
\]

The jeu de taquin algorithm yields $e_{\lambda_{12},\mu_{12}}^{\nu_{12}}=2$, $e_{\lambda_{13},\mu_{13}}^{\nu_{13}}=1$, $e_{\lambda_{23},\mu_{23}}^{\nu_{23}}=1$, hence
\[b_{\lambda,\mu}^{\nu}(Fl_{3,6; \mathbb{C}^7}) =  2\cdot 1\cdot 1=2.\]

\end{Example}

In contrast, for general Schubert structure constants not covered by Theorem~\ref{Thm:BKRYD} the regions are not independent. For example, let $n=5$ and ${\tt k} = \{2,4\}$.

\begin{Example}\label{ex:Schubertregions}
$\sigma_{12453}\cdot\sigma_{34125}=\sigma_{35142}+\sigma_{34251}+
\sigma_{45123}\in H^{\star}(Fl_{2,4;\mathbb{C}^5})$. Pictorially:

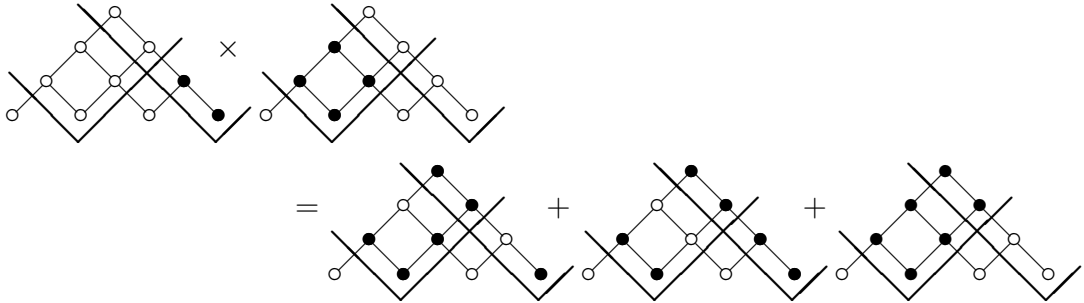
\[\begin{picture}(400, 110)
\multiput(0,65)(13,13){4}{$\circ$}
\multiput(26,65)(13,13){3}{$\circ$}
\multiput(52,65)(13,13){2}{$\circ$}
\multiput(78,65)(13,13){1}{$\circ$}

\thicklines
\put(28,58){\line(1,1){39}}
\put(28,58){\line(-1,1){26}}
\put(80,58){\line(1,1){13}}
\put(80,58){\line(-1,1){52}}
\thinlines

\multiput(4.5,69.5)(13,13){3}{\line(1,1){10}}
\multiput(30.5,69.5)(13,13){2}{\line(1,1){10}}
\multiput(56.5,69.5)(13,13){1}{\line(1,1){10}}

\multiput(27.5,69.5)(13,13){3}{\line(-1,1){10}}
\multiput(53.5,69.5)(13,13){2}{\line(-1,1){10}}
\multiput(79.5,69.5)(13,13){1}{\line(-1,1){10}}

\put(65,78){$\bullet$}
\put(78,65){$\bullet$}

\put(80,90){$\times$}


\multiput(96,65)(13,13){4}{$\circ$}
\multiput(122,65)(13,13){3}{$\circ$}
\multiput(148,65)(13,13){2}{$\circ$}
\multiput(174,65)(13,13){1}{$\circ$}

\thicklines
\put(124,58){\line(1,1){39}}
\put(124,58){\line(-1,1){26}}
\put(176,58){\line(1,1){13}}
\put(176,58){\line(-1,1){52}}
\thinlines

\multiput(109,78)(13,13){2}{$\bullet$}
\multiput(122,65)(13,13){2}{$\bullet$}

\multiput(100.5,69.5)(13,13){3}{\line(1,1){10}}
\multiput(126.5,69.5)(13,13){2}{\line(1,1){10}}
\multiput(152.5,69.5)(13,13){1}{\line(1,1){10}}

\multiput(123.5,69.5)(13,13){3}{\line(-1,1){10}}
\multiput(149.5,69.5)(13,13){2}{\line(-1,1){10}}
\multiput(175.5,69.5)(13,13){1}{\line(-1,1){10}}


\put(110,30){$=$}

\multiput(122,5)(13,13){4}{$\circ$}
\multiput(148,5)(13,13){3}{$\circ$}
\multiput(174,5)(13,13){2}{$\circ$}
\multiput(200,5)(13,13){1}{$\circ$}

\thicklines
\put(150,-2){\line(1,1){39}}
\put(150,-2){\line(-1,1){26}}
\put(202,-2){\line(1,1){13}}
\put(202,-2){\line(-1,1){52}}
\thinlines

\multiput(135,18)(13,13){1}{$\bullet$}
\put(161,44){$\bullet$}
\multiput(148,5)(13,13){3}{$\bullet$}
\put(200,5){$\bullet$}

\multiput(126.5,9.5)(13,13){3}{\line(1,1){10}}
\multiput(152.5,9.5)(13,13){2}{\line(1,1){10}}
\multiput(178.5,9.5)(13,13){1}{\line(1,1){10}}

\multiput(149.5,9.5)(13,13){3}{\line(-1,1){10}}
\multiput(175.5,9.5)(13,13){2}{\line(-1,1){10}}
\multiput(201.5,9.5)(13,13){1}{\line(-1,1){10}}


\put(205,30){$+$}

\multiput(218,5)(13,13){4}{$\circ$}
\multiput(244,5)(13,13){3}{$\circ$}
\multiput(270,5)(13,13){2}{$\circ$}
\multiput(296,5)(13,13){1}{$\circ$}

\thicklines
\put(246,-2){\line(1,1){39}}
\put(246,-2){\line(-1,1){26}}
\put(298,-2){\line(1,1){13}}
\put(298,-2){\line(-1,1){52}}
\thinlines

\multiput(231,18)(13,13){1}{$\bullet$}
\put(257,44){$\bullet$}
\multiput(244,5)(13,13){1}{$\bullet$}
\put(270,31){$\bullet$}
\put(283,18){$\bullet$}
\put(296,5){$\bullet$}

\multiput(222.5,9.5)(13,13){3}{\line(1,1){10}}
\multiput(248.5,9.5)(13,13){2}{\line(1,1){10}}
\multiput(274.5,9.5)(13,13){1}{\line(1,1){10}}

\multiput(245.5,9.5)(13,13){3}{\line(-1,1){10}}
\multiput(271.5,9.5)(13,13){2}{\line(-1,1){10}}
\multiput(297.5,9.5)(13,13){1}{\line(-1,1){10}}


\put(302,30){$+$}

\multiput(327,18)(13,13){3}{$\bullet$}
\multiput(314,5)(13,13){1}{$\circ$}
\multiput(340,5)(13,13){3}{$\bullet$}
\multiput(366,5)(13,13){2}{$\circ$}
\multiput(392,5)(13,13){1}{$\circ$}

\thicklines
\put(342,-2){\line(1,1){39}}
\put(342,-2){\line(-1,1){26}}
\put(394,-2){\line(1,1){13}}
\put(394,-2){\line(-1,1){52}}
\thinlines

\multiput(318.5,9.5)(13,13){3}{\line(1,1){10}}
\multiput(344.5,9.5)(13,13){2}{\line(1,1){10}}
\multiput(370.5,9.5)(13,13){1}{\line(1,1){10}}

\multiput(341.5,9.5)(13,13){3}{\line(-1,1){10}}
\multiput(367.5,9.5)(13,13){2}{\line(-1,1){10}}
\multiput(393.5,9.5)(13,13){1}{\line(-1,1){10}}

\end{picture}
\]

\end{Example}

The RYDs for $12453$ and $34125$ use no roots from $\Lambda_{{\tt k}}^{13}$, but the RYDs for $35142, 34251$ and $45123$ all use roots from this region. In particular, by Theorem~\ref{Thm:BKRYD} this immediately implies $\sigma_{12453}\odot_0\sigma_{34125}=0$.

\begin{Example}\label{example:KnPu11}
For purposes of comparison, we compute the example of \cite[Figure 2]{Knutson.Purbhoo} in terms of RYDs. Let $n=5$ and ${\tt k} = \{2,4\}$. Let $23112$, $12132$, $32121 \in G^{{\tt k}}_5$. Their images under $f$ are respectively $34152$, $13254$, $35241 \in S^{{\tt k}}_5$. Let respectively $\lambda$, $\mu$, $\nu \in \mathbb{Y}_{{\tt k}}$ be the corresponding RYDs. Below is the only possible set of labellings $\{T_{\mu_{ij}}\}$ of the RYD $\mu$, and the only possible set of labellings $\{T_{\nu_{ij}/\lambda_{ij}}\}$ of the skew RYD $\nu/\lambda$. 

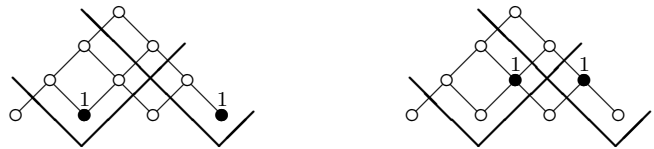
\[\begin{picture}(240,60)

\multiput(0,13)(13,13){4}{$\circ$}
\multiput(26,13)(13,13){3}{$\circ$}
\multiput(52,13)(13,13){2}{$\circ$}
\multiput(78,13)(13,13){1}{$\circ$}

\put(26,13){$\bullet$}
\put(78,13){$\bullet$}

\put(27,19.5){{\tiny $1$}}
\put(79,19.5){{\tiny $1$}}

\thicklines
\put(28,4){\line(1,1){39}}
\put(28,4){\line(-1,1){26}}
\put(80,4){\line(1,1){13}}
\put(80,4){\line(-1,1){52}}
\thinlines

\multiput(4.5,17.5)(13,13){3}{\line(1,1){10}}
\multiput(30.5,17.5)(13,13){2}{\line(1,1){10}}
\multiput(56.5,17.5)(13,13){1}{\line(1,1){10}}

\multiput(27.5,17.5)(13,13){3}{\line(-1,1){10}}
\multiput(53.5,17.5)(13,13){2}{\line(-1,1){10}}
\multiput(79.5,17.5)(13,13){1}{\line(-1,1){10}}


\multiput(150,13)(13,13){4}{$\circ$}
\multiput(176,13)(13,13){3}{$\circ$}
\multiput(202,13)(13,13){2}{$\circ$}
\multiput(228,13)(13,13){1}{$\circ$}

\put(189,26){$\bullet$}
\put(215,26){$\bullet$}

\put(190,32.5){{\tiny $1$}}
\put(216,32.5){{\tiny $1$}}

\thicklines
\put(178,4){\line(1,1){39}}
\put(178,4){\line(-1,1){26}}
\put(230,4){\line(1,1){13}}
\put(230,4){\line(-1,1){52}}
\thinlines

\multiput(154.5,17.5)(13,13){3}{\line(1,1){10}}
\multiput(180.5,17.5)(13,13){2}{\line(1,1){10}}
\multiput(206.5,17.5)(13,13){1}{\line(1,1){10}}

\multiput(177.5,17.5)(13,13){3}{\line(-1,1){10}}
\multiput(203.5,17.5)(13,13){2}{\line(-1,1){10}}
\multiput(229.5,17.5)(13,13){1}{\line(-1,1){10}}

\end{picture}\]

Since ${\tt jdt}(T_{\nu_{ij}/\lambda_{ij}})=T_{\mu_{ij}}$ in each region $\Lambda_{{\tt k}}^{ij}$, we have $b_{\lambda,\mu}^\nu(Fl_{2,4; \mathbb{C}^5})=1$.

\end{Example}

The Belkale-Kumar product has recently been utilized to obtain results concerning the structure of the Littlewood-Richardson cone and generalizations thereof (\cite{RessayreGIT}, \cite{RessayreGIT2}). Fulton's conjecture, proved in \cite{Knutson.Tao.Woodward} (also geometrically in \cite{Belkale} and \cite{RessayreFulton}) has also been generalized by \cite{Belkale.Kumar.Ressayre} using the Belkale-Kumar product. Further applications include the Horn problem (\cite{Richmond}, \cite{Belkale.Kumar10}, \cite{RessayreEigencone}), and branching Schubert calculus (\cite{Ressayre.Richmond}).

\subsection{Nonmaximal isotropic Grassmannians}
Fix a positive integer $k<n$. A (nonmaximal) isotropic Grassmannian is the set of $k$-dimensional isotropic subspaces of a vector space with a non-degenerate symmetric or skew-symmetric bilinear form. 
Specifically, they are the odd orthogonal Grassmannian $OG(k,2n+1)$, the Lagrangian Grassmannian $LG(k,2n)$, and the even orthogonal Grassmannian $OG(k,2n)$.

The Schubert varieties of $OG(k,2n+1)$ and $LG(k,2n)$ are both indexed by a set denoted $W^{OG(k,2n+1)}$, and the Schubert varieties of $OG(k,2n)$ are indexed by a set $W^{OG(k,2n)}$. The elements of these sets are certain signed permutations corresponding to Weyl group cosets, and are described explicitly in Section~3.
For $OG(k,2n+1)/LG(k,2n)$ (respectively, $OG(k,2n)$), the RYDs of \cite{Searles.Yong} are the inversion sets of the elements of $W^{OG(k,2n+1)}$ (respectively, $W^{OG(k,2n)}$) in the type B root poset $\Omega_{SO_{2n+1}}$ (respectively, type D root poset $\Omega_{SO_{2n}}$). Denote the set of RYDs associated to $W^{OG(k,2n+1)}$ (respectively, $W^{OG(k,2n)}$) by  $\mathbb{Y}_{OG(k,2n+1)}$ (respectively, $\mathbb{Y}_{OG(k,2n)}$).

\begin{Example}\label{ex:Bpicture} 
Below are two RYDs shown inside $\Omega_{SO_{11}}$. The first is an element of $\mathbb{Y}_{OG(3,11)}$, the second an element of $\mathbb{Y}_{OG(4,11)}$.

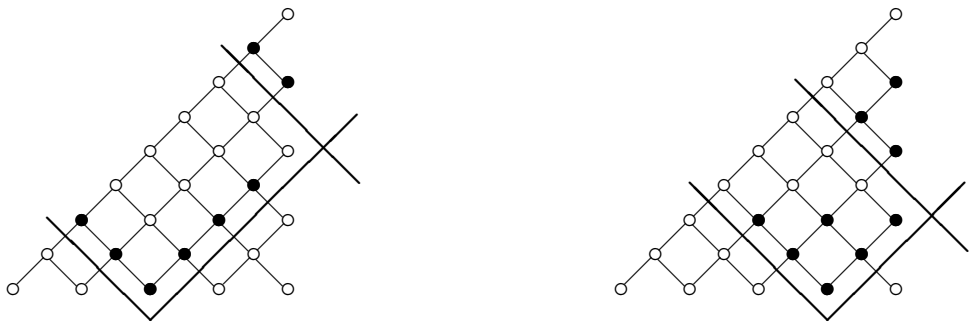
\[\begin{picture}(360,130)

\multiput(0,13)(13,13){9}{$\circ$}
\multiput(26,13)(13,13){7}{$\circ$}
\multiput(52,13)(13,13){5}{$\circ$}
\multiput(78,13)(13,13){3}{$\circ$}
\multiput(104,13)(13,13){1}{$\circ$}

\multiput(52,13)(13,13){4}{$\bullet$}
\multiput(39,26)(13,13){1}{$\bullet$}
\multiput(26,39)(13,13){1}{$\bullet$}
\put(91,104){$\bullet$}
\put(104,91){$\bullet$}

\multiput(4.5,17.5)(13,13){8}{\line(1,1){10}}
\multiput(30.5,17.5)(13,13){6}{\line(1,1){10}}
\multiput(56.5,17.5)(13,13){4}{\line(1,1){10}}
\multiput(82.5,17.5)(13,13){2}{\line(1,1){10}}

\multiput(17,27)(13,13){7}{\line(1,-1){10}}
\multiput(43,27)(13,13){5}{\line(1,-1){10}}
\multiput(69,27)(13,13){3}{\line(1,-1){10}}
\multiput(95,27)(13,13){1}{\line(1,-1){10}}

\thicklines
\put(55,4){\line(-1,1){39}}
\put(55,4){\line(1,1){78}}
\put(134,56){\line(-1,1){52}}
\thinlines


\multiput(230,13)(13,13){9}{$\circ$}
\multiput(256,13)(13,13){7}{$\circ$}
\multiput(282,13)(13,13){5}{$\circ$}
\multiput(308,13)(13,13){3}{$\circ$}
\multiput(334,13)(13,13){1}{$\circ$}

\multiput(308,13)(13,13){3}{$\bullet$}
\multiput(295,26)(13,13){2}{$\bullet$}
\multiput(282,39)(13,13){1}{$\bullet$}

\put(334,65){$\bullet$}
\multiput(321,78)(13,13){2}{$\bullet$}

\multiput(234.5,17.5)(13,13){8}{\line(1,1){10}}
\multiput(260.5,17.5)(13,13){6}{\line(1,1){10}}
\multiput(286.5,17.5)(13,13){4}{\line(1,1){10}}
\multiput(312.5,17.5)(13,13){2}{\line(1,1){10}}

\multiput(247,27)(13,13){7}{\line(1,-1){10}}
\multiput(273,27)(13,13){5}{\line(1,-1){10}}
\multiput(299,27)(13,13){3}{\line(1,-1){10}}
\multiput(325,27)(13,13){1}{\line(1,-1){10}}

\thicklines
\put(311,4){\line(-1,1){52}}
\put(311,4){\line(1,1){52}}
\put(364,30){\line(-1,1){65}}
\thinlines

\end{picture}\]
\end{Example}

\begin{Example}\label{ex:Dpicture} 
Below are two RYDs shown inside $\Omega_{SO_{12}}$. The first is an element of $\mathbb{Y}_{OG(3,12)}$, and also shown is a ``double-tailed diamond'' from its base region (see the explanation below). The second is an element of $\mathbb{Y}_{OG(4,12)}$.

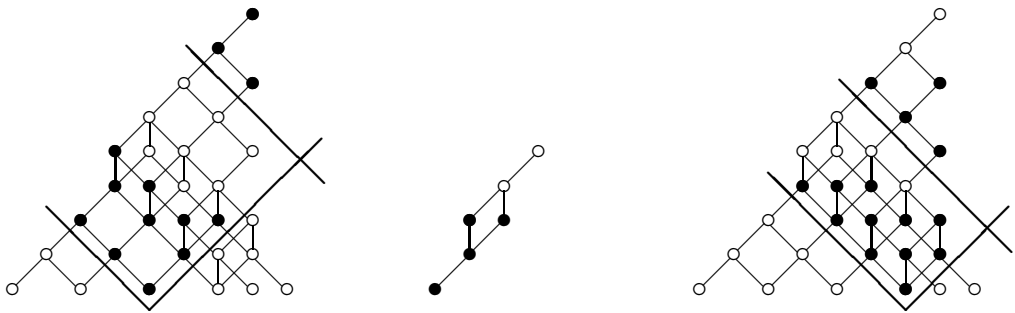
\[\begin{picture}(400,130)

\multiput(0,13)(13,13){5}{$\circ$}
\multiput(26,13)(13,13){4}{$\circ$}
\multiput(52,13)(13,13){3}{$\circ$}
\multiput(78,13)(13,13){2}{$\circ$}
\multiput(104,13)(13,13){1}{$\circ$}

\multiput(4.5,17.5)(13,13){4}{\line(1,1){10}}
\multiput(30.5,17.5)(13,13){3}{\line(1,1){10}}
\multiput(56.5,17.5)(13,13){2}{\line(1,1){10}}
\multiput(82.5,17.5)(13,13){1}{\line(1,1){10}}

\multiput(17,27)(13,13){4}{\line(1,-1){10}}
\multiput(43,27)(13,13){3}{\line(1,-1){10}}
\multiput(69,27)(13,13){2}{\line(1,-1){10}}
\multiput(95,27)(13,13){1}{\line(1,-1){10}}

\multiput(52,13)(13,13){3}{$\bullet$}
\multiput(39,26)(13,13){2}{$\bullet$}
\multiput(26,39)(13,13){2}{$\bullet$}
\put(78,104){$\bullet$}
\put(91,91){$\bullet$}
\put(91,117){$\bullet$}

\multiput(39,65)(13,13){5}{$\circ$}
\multiput(52,52)(13,13){4}{$\circ$}
\multiput(65,39)(13,13){3}{$\circ$}
\multiput(78,26)(13,13){2}{$\circ$}
\multiput(91,13)(13,13){1}{$\circ$}

\multiput(43.5,69.5)(13,13){4}{\line(1,1){10}}
\multiput(56.5,56.5)(13,13){3}{\line(1,1){10}}
\multiput(69.5,43.5)(13,13){2}{\line(1,1){10}}
\multiput(82.5,30.5)(13,13){1}{\line(1,1){10}}

\multiput(43,66)(13,13){4}{\line(1,-1){10}}
\multiput(56,53)(13,13){3}{\line(1,-1){10}}
\multiput(69,40)(13,13){2}{\line(1,-1){10}}
\multiput(82,27)(13,13){1}{\line(1,-1){10}}

\put(39,65){$\bullet$}
\put(52,52){$\bullet$}
\put(65,39){$\bullet$}

\multiput(42,57)(13,13){2}{\line(0,1){9}}
\multiput(55,44)(13,13){2}{\line(0,1){9}}
\multiput(68,31)(13,13){2}{\line(0,1){9}}
\multiput(81,18)(13,13){2}{\line(0,1){9}}

\thicklines
\put(55,8){\line(-1,1){39}}
\put(55,8){\line(1,1){65}}
\put(121,56){\line(-1,1){52}}
\thinlines


\multiput(160,13)(13,13){3}{$\bullet$}
\multiput(173,39)(13,13){3}{$\circ$}
\put(173,39){$\bullet$}

\multiput(164.5,17.5)(13,13){2}{\line(1,1){10}}
\multiput(177.5,43.5)(13,13){2}{\line(1,1){10}}

\put(176,31){\line(0,1){9}}
\put(189,44){\line(0,1){9}}



\multiput(260,13)(13,13){5}{$\circ$}
\multiput(286,13)(13,13){4}{$\circ$}
\multiput(312,13)(13,13){3}{$\circ$}
\multiput(338,13)(13,13){2}{$\circ$}
\multiput(364,13)(13,13){1}{$\circ$}

\multiput(264.5,17.5)(13,13){4}{\line(1,1){10}}
\multiput(290.5,17.5)(13,13){3}{\line(1,1){10}}
\multiput(316.5,17.5)(13,13){2}{\line(1,1){10}}
\multiput(342.5,17.5)(13,13){1}{\line(1,1){10}}

\multiput(277,27)(13,13){4}{\line(1,-1){10}}
\multiput(303,27)(13,13){3}{\line(1,-1){10}}
\multiput(329,27)(13,13){2}{\line(1,-1){10}}
\multiput(355,27)(13,13){1}{\line(1,-1){10}}

\multiput(338,13)(13,13){2}{$\bullet$}
\multiput(325,26)(13,13){2}{$\bullet$}
\multiput(312,39)(13,13){2}{$\bullet$}
\multiput(299,52)(13,13){1}{$\bullet$}

\put(325,91){$\bullet$}
\multiput(338,78)(13,13){2}{$\bullet$}
\put(351,65){$\bullet$}


\multiput(299,65)(13,13){5}{$\circ$}
\multiput(312,52)(13,13){4}{$\circ$}
\multiput(325,39)(13,13){3}{$\circ$}
\multiput(338,26)(13,13){2}{$\circ$}
\multiput(351,13)(13,13){1}{$\circ$}

\multiput(303.5,69.5)(13,13){4}{\line(1,1){10}}
\multiput(316.5,56.5)(13,13){3}{\line(1,1){10}}
\multiput(329.5,43.5)(13,13){2}{\line(1,1){10}}
\multiput(342.5,30.5)(13,13){1}{\line(1,1){10}}

\multiput(303,66)(13,13){4}{\line(1,-1){10}}
\multiput(316,53)(13,13){3}{\line(1,-1){10}}
\multiput(329,40)(13,13){2}{\line(1,-1){10}}
\multiput(342,27)(13,13){1}{\line(1,-1){10}}

\multiput(338,26)(13,13){2}{$\bullet$}
\multiput(325,39)(13,13){1}{$\bullet$}
\multiput(312,52)(13,13){1}{$\bullet$}

\multiput(302,57)(13,13){2}{\line(0,1){9}}
\multiput(315,44)(13,13){2}{\line(0,1){9}}
\multiput(328,31)(13,13){2}{\line(0,1){9}}
\multiput(341,18)(13,13){2}{\line(0,1){9}}

\thicklines
\put(341,8){\line(-1,1){52}}
\put(341,8){\line(1,1){39}}
\put(381,30){\line(-1,1){65}}
\thinlines

\end{picture}\]
\end{Example}

We now explain the diagrams of Examples~\ref{ex:Bpicture} and \ref{ex:Dpicture} above. Let $\{\beta_1,\ldots, \beta_n\}$ denote the roots of the standard embedding of the type $B_n$ (respectively, $D_n$) root system into $\mathbb{R}^n$. Then every RYD in $\mathbb{Y}_{OG(k,2n+1)}$ (respectively, $\mathbb{Y}_{OG(k,2n)}$) is in fact contained in the subposet $\Lambda_k$ of $\Omega_{SO_{2n+1}}$ (respectively, $\Omega_{SO_{2n}}$) consisting of all roots above the $k$th simple root $\beta_k$. 
We divide $\Lambda_{k}$ into a {\bf base region} and a {\bf top region}. In Examples~\ref{ex:Bpicture} and \ref{ex:Dpicture}, the thicker black lines show $\Lambda_{k}$ and its division into these two regions.
In each type, the top region is a ``staircase'' $(k-1,k-2,\ldots,0)$. In types B/C the base region is a $k\times (2n+1-2k)$ ``rectangle'', while in type D the base region consists of $k$ ``double-tailed diamonds'' (following the nomenclature of \cite{Thomas.Yong:comin}) each having $2n-2k$ roots.

It is straightforward to show that every RYD $\lambda$ consists of a lower order ideal in each region. Then an RYD $\lambda$ for a nonmaximal isotropic Grassmannian has a natural visual interpretation as a pair of partitions $(\lambda^{(1)}|\lambda^{(2)})$, corresponding to the base and top regions. This allows us to write the RYDs in a compact way. Pairs of partitions are used in other indexing sets for Schubert varieties for these spaces, see, e.g., \cite{Prag}, \cite{PragD}, \cite{Tam:qcig}, \cite{Coskun.Vakil}, \cite{Coskun.Orthogonal}, but the pairs of partitions used in these indexing sets differ from those that arise from RYDs. 

We now describe the pair of partitions $(\lambda^{(1)}|\lambda^{(2)})$ associated to an RYD $\lambda$. In each type, $\lambda^{(2)}$ is a strict partition in $(k-1,k-2,\ldots,0)$. In types B/C, $\lambda^{(1)}$ is a partition in $k\times (2n+1-2k)$. In type D, $\lambda^{(1)}$ is a partition in $k\times (2n-2k)$, and also if $\lambda^{(1)}_i=n-k$ for some $1\le i\le k$ we assign a $\uparrow$ (respectively, $\downarrow$) if $\lambda$ uses the root above $\beta_{n-1}$ (respectively, $\beta_n$) in the $i$th double-tailed diamond. 

\begin{Example}
In the partition pair notation, the RYDs of Example~\ref{ex:Bpicture} are respectively $((4,1,1)|(2,0,0))$ and $((3,2,1,0)|(2,1,0,0))$, and the RYDs of Example~\ref{ex:Dpicture} are respectively $((4,3,3)|(2,1,0))^\uparrow$ and $((4,3,3,1)|(3,1,0,0))$. 
\end{Example}

We now follow \cite{BKT:Inventiones}. An {\bf $(n-k)$-strict partition} is defined to be a partition $\gamma$ such that $\gamma_i>\gamma_{i+1}$ whenever $\gamma_i > n-k$. The Schubert varieties of $OG(k,2n+1)$ and $LG(k,2n)$ are indexed by the set $P(n-k,n)$ of all $(n-k)$-strict partitions in a $k \times (2n-k)$ rectangle. The Schubert varieties of $OG(k,2n)$ are indexed by the set $\tilde{P}(n-k,n)$ of all pairs $\tilde{\gamma}=(\gamma; {\tt type}(\gamma))$, where $\gamma$ is an $(n-k)$-strict partition in a $k \times (2n-1-k)$ rectangle, and also ${\tt type}(\gamma)=0$ if no part of $\gamma$ has size $n-k$ and ${\tt type}(\gamma)\in \{1,2\}$ otherwise. 

We obtain the following translations between RYDs and the indexing sets of \cite{BKT:Inventiones}:

\begin{Proposition}\label{prop:RYDtoBKT}
There is a bijection $f_k:\mathbb{Y}_{OG(k,2n+1)}\rightarrow P(n-k,n)$ for each $1\le k<n$, via
\[f_k(\lambda) = (\lambda^{(1)}_i+\lambda^{(2)}_i)_{1\le i\le k}.\]
The Schubert variety indexed by $\lambda$ is equal to the Schubert variety indexed by $f_k(\lambda)$.
\end{Proposition}

\begin{Example}
The RYDs of Example~\ref{ex:Bpicture} correspond respectively to $(6,1,1)\in P(2,5)$ and $(5,3,1)\in P(1,5)$.
\end{Example}

\begin{Proposition}\label{prop:RYDtoBKTD}
There is a bijection $F_k:\mathbb{Y}_{OG(k,2n)}\rightarrow \tilde{P}(n-k,n)$ for each $1\le k <n$, via 
\[F_k(\lambda) = \begin{cases}
((\lambda^{(1)}_i+\lambda^{(2)}_i)_{1\le i\le k};1) & \text{if $\lambda$ is assigned $\uparrow$}\\                 
((\lambda^{(1)}_i+\lambda^{(2)}_i)_{1\le i\le k};2) & \text{if $\lambda$ is assigned $\downarrow$}\\                 
((\lambda^{(1)}_i+\lambda^{(2)}_i)_{1\le i\le k};0) & \text{otherwise}      
\end{cases}\]
The Schubert variety indexed by $\lambda$ is equal to the Schubert variety indexed by $F_k(\lambda)$.
\end{Proposition}

\begin{Example}
The RYDs of Example~\ref{ex:Dpicture} correspond respectively to $((6,4,3);1)\in P(3,6)$ and $((7,4,3,1);0)\in P(2,6)$.
\end{Example}

Propositions~\ref{prop:RYDtoBKT} and \ref{prop:RYDtoBKTD} are used to prove agreement of \cite[Theorem 4.1]{Searles.Yong} and \cite[Theorem 5.3]{Searles.Yong} with the Pieri rules of \cite{BKT:Inventiones}. Specifically, let 
$\star$ denote the product on RYDs of \cite[Theorem 4.1]{Searles.Yong} or \cite[Theorem 5.3]{Searles.Yong}, and let $\Psi$ denote the linear map determined by sending an RYD $\lambda$ to its corresponding Schubert class $\sigma_\lambda$. 

\begin{Theorem}\label{Thm:Pieriagreement}
Suppose $\lambda$ is an RYD indexing a Pieri class. Then
\begin{itemize}
\item[(I)] If $\lambda,\mu\in \mathbb{Y}_{OG(2,2n+1)}$, then $\Psi(\lambda\star\mu) = \sigma_{f_2(\lambda)}\cdot\sigma_{f_2(\mu)}\in H^\star(LG(2,2n))$
\item[(II)] If $\lambda,\mu\in \mathbb{Y}_{OG(2,2n)}$, then $\Psi(\lambda\star\mu) = \sigma_{F_2(\lambda)}\cdot\sigma_{F_2(\mu)}\in H^\star(OG(2,2n))$.
\end{itemize}
\end{Theorem}

\subsection{Organization}
In Section 2, we prove Theorem~\ref{Thm:BKRYD}. In Section 3, we prove Propositions~\ref{prop:RYDtoBKT} and \ref{prop:RYDtoBKTD}. In Section 4, we utilize Proposition~\ref{prop:RYDtoBKT} to prove Theorem~\ref{Thm:Pieriagreement}(I), and in Section 5 we utilize Proposition~\ref{prop:RYDtoBKT} to prove Theorem~\ref{Thm:Pieriagreement}(II). As discussed in \cite{Searles.Yong}, the correctness of the rule for $LG(2,2n)$ implies the correctness of the rule for the adjoint $OG(2,2n+1)$, so we do not need to prove this case separately.

\section{Proof of Theorem~\ref{Thm:BKRYD}}

We begin by completely characterizing the RYDs of \cite{Searles.Yong} in the case of $F_{{\tt k}}$. Fix ${\tt k}$ and recall $I_i=[k_{i-1}+1,k_i]$ for $1\le i\le d$, where we set $k_0=0$ and $k_d=n$. Let $C$ denote the set of all nonnegative integer vectors $c=(c_1,\ldots,c_{n-1})$ satisfying $c_j\le n-j$.  Let $C_{{\tt k}}\subset C$ denote the set of $c\in C$ such that for $1\le j <n$, $c_j > c_{j+1}$ only if $j$ and $j+1$ are not in the same interval $I_i$ (we set $c_n=0$). For any permutation $w\in S_n$, its {\bf code} is defined to be the vector $c_w\in C$ such that $(c_w)_i$ is the number of positions $j$ satisfying $i<j$ and $w(i)>w(j)$. For example, if $n=7$ and ${\tt k} = \{1,3,5,6\}$ then $w = 5361742 \in S_7^{{\tt k}}$ has code $c_w=(4,2,3,0,2,1)\in C_{{\tt k}}$. The following is clear:

\begin{Claim}\label{claim:permtofromcode}
The map that takes $w\in S_n^{{\tt k}}$ to its code $c_w$ is a bijection $S_n^{{\tt k}}\rightarrow C_{{\tt k}}$.
\end{Claim}

Given a subset $S\subset\Omega_{GL_n}$, define a nonnegative integer vector $h_S=(h_1,\ldots h_{n-1})$ by letting $h_j$ be the number of roots of the form $(j,b)=e_j-e_b$ in $S$. Call $S$ a {\bf ${\tt k}$-diagram} if the roots in $S$ form a lower order ideal in each region $\Lambda_{{\tt k}}^{ij}$, and also $S$ satisfies a {\bf hook condition}: a root $\alpha$ must be in $S$ (respectively, must not be in $S$) if more than half of the roots in $\Omega_{GL_n}$ diagonally south-east and south-west of $\alpha$ are in $S$ (respectively, not in $S$). Let $\Theta_{{\tt k}}$ denote the set of all ${\tt k}$-diagrams. 

\begin{Claim}\label{claim:diagramtocode}
The map that takes a ${\tt k}$-diagram $\theta$ to $h_\theta$ is an injection $\Theta_{{\tt k}}\rightarrow C_{{\tt k}}$.
\end{Claim}
\begin{proof}
By definition, $h_j\le n-j$. The condition that the roots in $\theta$ form a lower order ideal in each region forces $h_j > h_{j+1}$ only if $j$ and $j+1$ are not in the same interval $I_i$. So $h_\theta\in C_{{\tt k}}$.

To show injectivity, we will show that given $c\in C$, there is a unique $S\subset \Omega_{GL_n}$ satisfying both $h_S=c$ and the hook condition. We construct $S$ by coloring a root of $\Omega_{GL_n}$ black if it is in $S$, and white if it is not in $S$. If $c_{n-1}=0$ then we must color the root $(n-1,n)$ white, and if $c_{n-1}=1$ we must color it black. Now proceed inductively. Fix $j<n-1$ and suppose all roots of the form $(a,b)$ with $a>j$ have been colored white or black. Use the following procedure to color roots of the form $(j,b)$ black one-by-one until $h_j$ such roots have been colored black, at which point terminate the procedure and color all remaining such roots white:

If there exists a root of the form $(j,b)$ such that exactly half of the roots diagonally south-east and south-west of it are colored black, then color the highest such root black. Otherwise, color the lowest root of the form $(j,b)$ black.

It is clear that each coloring of a root in the above procedure is forced by the hook condition. Therefore, since the elements of $\Theta_{{\tt k}}$ satisfy the hook condition, the map $\Theta_{{\tt k}}\rightarrow C_{{\tt k}}$ is injective.
\end{proof}

\begin{Example}
Suppose $c=(4,2,3,0,2,1)$. Then the unique $S$ satisfying $h_S=c$ and the hook condition is shown below, with the roots in $S$ labelled according to the order in which they were colored black by the procedure of Claim~\ref{claim:diagramtocode}.

\[\begin{picture}(120,85)
\multiput(4.5,17.5)(13,13){5}{\line(1,1){10}}
\multiput(0,13)(13,13){6}{$\circ$}
\multiput(26,13)(13,13){5}{$\circ$}
\multiput(52,13)(13,13){4}{$\circ$}
\multiput(78,13)(13,13){3}{$\circ$}
\multiput(104,13)(13,13){2}{$\circ$}
\multiput(130,13)(13,13){1}{$\circ$}
\multiput(30.5,17.5)(13,13){4}{\line(1,1){10}}
\multiput(56.5,17.5)(13,13){3}{\line(1,1){10}}
\multiput(82.5,17.5)(13,13){2}{\line(1,1){10}}
\multiput(108.5,17.5)(13,13){1}{\line(1,1){10}}
\multiput(17,27)(13,13){5}{\line(1,-1){10}}
\multiput(43,27)(13,13){4}{\line(1,-1){10}}
\multiput(69,27)(13,13){3}{\line(1,-1){10}}
\multiput(95,27)(13,13){2}{\line(1,-1){10}}
\multiput(121,27)(13,13){1}{\line(1,-1){10}}

\put(130,13){$\bullet$}
\put(131,19.5){{\tiny $1$}}

\put(117,26){$\bullet$}
\put(118,32.5){{\tiny $2$}}

\put(104,13){$\bullet$}
\put(105,19.5){{\tiny $3$}}

\put(52,13){$\bullet$}
\put(53,19.5){{\tiny $4$}}

\put(91,52){$\bullet$}
\put(92,58.5){{\tiny $5$}}

\put(78,39){$\bullet$}
\put(79,45.5){{\tiny $6$}}

\put(39,26){$\bullet$}
\put(40,32.5){{\tiny $7$}}

\put(78,65){$\bullet$}
\put(79,71.5){{\tiny $8$}}

\put(26,39){$\bullet$}
\put(27,45.5){{\tiny $9$}}

\put(65,78){$\bullet$}
\put(63,84.5){{\tiny $10$}}

\put(0,13){$\bullet$}
\put(-2.5,19.5){{\tiny $11$}}

\put(52,65){$\bullet$}
\put(49.5,71.5){{\tiny $12$}}

\end{picture}\]
\end{Example}

\begin{Claim}\label{claim:permtodiagram}
$\mathbb{Y}_{{\tt k}}\subseteq\Theta_{\tt k}$.
\end{Claim}
\begin{proof}
Let $\lambda\in \mathbb{Y}_{{\tt k}}$ and let $w$ be the element of $S_n^{{\tt k}}$ corresponding to $\lambda$. Consider a region $\Lambda_{{\tt k}}^{ij}$ of $\Omega_{GL_n}$. Let $a,a'\in I_i$ and $b,b'\in I_j$, and suppose $(a',b')\preceq (a,b)$ in $\Omega_{GL_n}$. Then by definition, $a\le a'$ and $b'\le b$. If also $w(a)>w(b)$, then since $w$ is increasing on $I_i$ and $I_j$, we have $w(a')>w(b')$. Thus the restriction $\lambda_{ij}$ of $\lambda$ to $\Lambda_{{\tt k}}^{ij}$ is a lower order ideal in $\Lambda_{{\tt k}}^{ij}$.

Now consider any root $(a,b)\in\Omega_{GL_n}$. The hook associated to $(a,b)$ is all roots $(a,l)$ for $a<l<b$ and all roots $(j,b)$ for $a<j<b$. If more than half of these are inverted by $w$, then there exists an $m$ with $a<m<b$ such that $w(a)>w(m)$ and $w(m)>w(b)$, hence $w$ must invert $(a,b)$. Similarly, if fewer than half of the roots in the hook are inverted, then $w$ cannot invert $(a,b)$. Thus $\lambda$ satisfies the hook condition.
\end{proof}

\begin{Corollary}\label{claim:RYDclassify}
$\mathbb{Y}_{{\tt k}} = \Theta_{{\tt k}}$.
\end{Corollary}
\begin{proof}
Composing the injection from Claim~\ref{claim:diagramtocode} with the bijection of Claim~\ref{claim:permtofromcode} yields an injection $\Theta_{{\tt k}}\rightarrow S_n^{{\tt k}}$. By definition $\mathbb{Y}_{{\tt k}}$ is in bijection with $S_n^{{\tt k}}$, thus we have an injection $\Theta_{{\tt k}}\rightarrow \mathbb{Y}_{{\tt k}}$. By Claim~\ref{claim:permtodiagram}, $\mathbb{Y}_{{\tt k}}\subseteq \Theta_{{\tt k}}$, so $\mathbb{Y}_{{\tt k}}= \Theta_{{\tt k}}$. 
\end{proof}

Let $r_i = |I_i|=k_i-k_{i-1}$. We now follow \cite{Knutson.Purbhoo}. 
Let $G_n^{{\tt k}}$ denote the set of $n$-letter words $\tau$ from the alphabet $\{1,\ldots,d\}$, such that the letter $i$ is used $r_i$ times in $\tau$. Then the Schubert varieties of $F_{{\tt k}}$ are indexed by the elements of $G_n^{{\tt k}}$. Define a map $f:G_n^{{\tt k}}\rightarrow S_n^{{\tt k}}$ by letting $f(\tau)$ be the permutation, in one-line notation, obtained by writing down the positions of the ones in order, then the positions of the twos in order, etc. For example, if ${\tt k} = \{3,5,6\}$ and $\tau=2431121\in G_7^{\tt k}$ then $f(\tau)=4571632\in S_7^{\tt k}$. This is a bijection, and the Schubert variety of $F_{{\tt k}}$ indexed by $\tau$ is equal to the Schubert variety indexed by $f(\tau)$. Given $i,j$ with $1\le i < j\le d$, let $D_{ij}(\tau)$ be the word obtained by deleting all letters of $\tau$ that are not $i$ or $j$. Then $D_{ij}(\tau)$ indexes a Schubert variety in the Grassmannian $Gr_{r_i}(\mathbb{C}^{r_i+r_j})$.

\begin{Theorem}\cite[Theorem 3]{Knutson.Purbhoo}\label{Thm:BKfactorization}
Let $\tau,\pi,\rho\in G_n^{{\tt k}}$. Then
\[b_{\tau,\pi}^{\rho}(F_{{\tt k}})=\prod_{1\le i<j\le d}C_{D_{ij}(\tau),D_{ij}(\pi)}^{D_{ij}(\rho)}(Gr_{r_i}(\mathbb{C}^{r_i+r_j})).\]
\end{Theorem}

Now let $w\in S_n^{{\tt k}}$. Define $D_{ij}'(w)$ to be the permutation on $[1,\ldots, r_i+r_j]$ whose entries are in the same relative order as the entries of the word obtained by deleting all entries of $w$ except those in $I_i$ or $I_j$. For example, let $n=7$, ${\tt k} = \{2,5\}$, and $w=2614537\in S_7^{{\tt k}}$. Then $D_{13}'(w) = 1324$, since deleting all entries of $w$ except those in $I_1$ or $I_3$ yields $2637$, which is in the same relative order as $1324$. This process is the same as in \cite[Definition 1]{Richmond}, where it is noted this is also the flattening function of \cite{Billey.Braden}. 

By definition, $D_{ij}'(w)\in S_{r_i+r_j}^{\{r_i\}}$. Thus $D_{ij}'(w)$ indexes a Schubert variety in the Grassmannian $Gr_{r_i}(\mathbb{C}^{r_i+r_j})$, and the RYD corresponding to $D_{ij}'(w)$ has only a single region inside $\Omega_{GL_{r_i+r_j}}$. We will denote this region $\Lambda_{r_i, r_i+r_j}$. Note that $\Lambda_{r_i, r_i+r_j}$ is the subposet of $\Omega_{GL_{r_i+r_j}}$ consisting of all roots above the $r_i$th simple root $e_{r_i}-e_{r_i+1}$.

\begin{Example}\label{ex:decomposition}
Let $n=7$ and ${\tt k} = \{2,5\}$. Then $r_1=2$, $r_2=3$ and $r_3=2$. Let $w=2614537\in S_7^{{\tt k}}$ and $\lambda$ the corresponding RYD. Below are $\lambda$ and the RYDs for, respectively, $D_{12}'(w) = 25134\in S_5^{\{2\}}$, $D_{13}'(w) = 1324\in S_4^{\{2\}}$ and $D_{23}'(w) = 13425\in S_5^{\{3\}}$. 

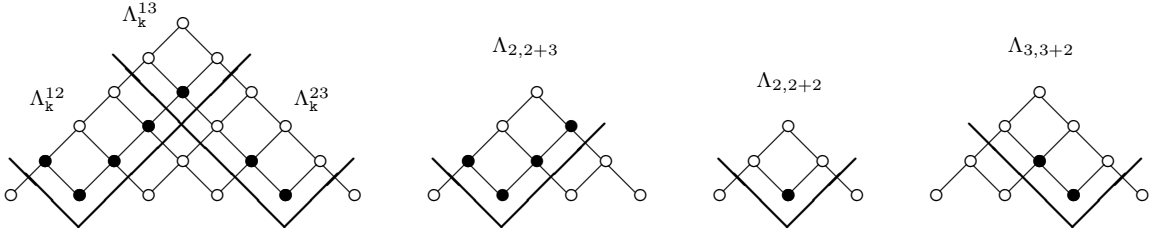
\[\begin{picture}(450,80)

\multiput(0,13)(13,13){6}{$\circ$}
\multiput(26,13)(13,13){5}{$\circ$}
\multiput(52,13)(13,13){4}{$\circ$}
\multiput(78,13)(13,13){3}{$\circ$}
\multiput(104,13)(13,13){2}{$\circ$}
\multiput(130,13)(13,13){1}{$\circ$}

\multiput(4.5,17.5)(13,13){5}{\line(1,1){10}}
\multiput(30.5,17.5)(13,13){4}{\line(1,1){10}}
\multiput(56.5,17.5)(13,13){3}{\line(1,1){10}}
\multiput(82.5,17.5)(13,13){2}{\line(1,1){10}}
\multiput(108.5,17.5)(13,13){1}{\line(1,1){10}}
\multiput(17,27)(13,13){5}{\line(1,-1){10}}
\multiput(43,27)(13,13){4}{\line(1,-1){10}}
\multiput(69,27)(13,13){3}{\line(1,-1){10}}
\multiput(95,27)(13,13){2}{\line(1,-1){10}}
\multiput(121,27)(13,13){1}{\line(1,-1){10}}

\put(13,26){$\bullet$}
\multiput(26,13)(13,13){4}{$\bullet$}

\put(104,13){$\bullet$}
\put(91,26){$\bullet$}

\thicklines

\put(28.5,4){\line(-1,1){26}}
\put(28.5,4){\line(1,1){65}}

\put(106.5,4){\line(-1,1){65}}
\put(106.5,4){\line(1,1){26}}

\thinlines

\put(10,50){{\tiny $\Lambda^{12}_{{\tt k}}$}}
\put(45,82){{\tiny $\Lambda^{13}_{{\tt k}}$}}
\put(110,50){{\tiny $\Lambda^{23}_{{\tt k}}$}}


\multiput(160,13)(13,13){4}{$\circ$}
\multiput(186,13)(13,13){3}{$\bullet$}
\multiput(212,13)(13,13){2}{$\circ$}
\multiput(238,13)(13,13){1}{$\circ$}

\multiput(164.5,17.5)(13,13){3}{\line(1,1){10}}
\multiput(190.5,17.5)(13,13){2}{\line(1,1){10}}
\multiput(216.5,17.5)(13,13){1}{\line(1,1){10}}

\multiput(177,27)(13,13){3}{\line(1,-1){10}}
\multiput(203,27)(13,13){2}{\line(1,-1){10}}
\multiput(229,27)(13,13){1}{\line(1,-1){10}}

\put(173,26){$\bullet$}

\thicklines

\put(188.5,4){\line(-1,1){26}}
\put(188.5,4){\line(1,1){39}}

\thinlines

\put(185,70){{\tiny $\Lambda_{2,2+3}$}}


\multiput(268,13)(13,13){3}{$\circ$}
\multiput(294,13)(13,13){2}{$\circ$}
\multiput(320,13)(13,13){1}{$\circ$}

\multiput(272.5,17.5)(13,13){2}{\line(1,1){10}}
\multiput(298.5,17.5)(13,13){1}{\line(1,1){10}}
\multiput(285,27)(13,13){2}{\line(1,-1){10}}
\multiput(311,27)(13,13){1}{\line(1,-1){10}}

\put(294,13){$\bullet$}

\thicklines

\put(296.5,4){\line(-1,1){26}}
\put(296.5,4){\line(1,1){26}}

\thinlines

\put(285,57){{\tiny $\Lambda_{2,2+2}$}}


\multiput(350,13)(13,13){4}{$\circ$}
\multiput(376,13)(13,13){3}{$\circ$}
\multiput(402,13)(13,13){2}{$\circ$}
\multiput(428,13)(13,13){1}{$\circ$}

\multiput(354.5,17.5)(13,13){3}{\line(1,1){10}}
\multiput(380.5,17.5)(13,13){2}{\line(1,1){10}}
\multiput(406.5,17.5)(13,13){1}{\line(1,1){10}}
\multiput(367,27)(13,13){3}{\line(1,-1){10}}
\multiput(393,27)(13,13){2}{\line(1,-1){10}}
\multiput(419,27)(13,13){1}{\line(1,-1){10}}

\put(402,13){$\bullet$}
\put(389,26){$\bullet$}

\thicklines

\put(404.5,4){\line(-1,1){39}}
\put(404.5,4){\line(1,1){26}}

\thinlines

\put(380,70){{\tiny $\Lambda_{3,3+2}$}}

\end{picture}\]

\end{Example}

The following is clear from the definitions: 

\begin{Lemma}\label{lemma:Dijagreement}
Let $\tau\in G_n^{{\tt k}}$. Then $D_{ij}'(f(\tau)) = f(D_{ij}(\tau))$.
\end{Lemma}

\noindent \emph{Proof of Theorem~\ref{Thm:BKRYD}:}
Let $\lambda,\mu,\nu\in \mathbb{Y}_{{\tt k}}$, respectively corresponding to permutations $u,v,w\in S_n^{{\tt k}}$. By Theorem~\ref{Thm:BKfactorization} and Lemma~\ref{lemma:Dijagreement}, we have
\[b_{\lambda,\mu}^\nu(F_{{\tt k}})= b_{u,v}^{w}(F_{{\tt k}})=\prod_{1\le i<j\le d}C_{D_{ij}'(u),D_{ij}'(v)}^{D_{ij}'(w)}(Gr_{r_i}(\mathbb{C}^{r_i+r_j})).\]

Straightforwardly, $\Lambda^{ij}_{{\tt k}}\subset \Omega_{GL_n}$ is isomorphic (as a poset) to $\Lambda_{r_i, r_i+r_j}$, and the roots in $\Lambda^{ij}_{{\tt k}}$ inverted by $w$ correspond to the roots of $\Lambda_{r_i, r_i+r_j}$ inverted by $D_{ij}'(w)$ (as depicted in Example~\ref{ex:decomposition}). Jeu de taquin is known to compute the Schubert structure constants for Grassmannians (see, e.g., \cite{Thomas.Yong:comin} for this root-theoretic setting). Therefore, we have $C_{D_{ij}'(u),D_{ij}'(v)}^{D_{ij}'(w)}(Gr_{r_i}(\mathbb{C}^{r_i+r_j})) = e_{\lambda_{ij},\mu_{ij}}^{\nu_{ij}}$.
\qed

\section{Proof of Propositions~\ref{prop:RYDtoBKT} and \ref{prop:RYDtoBKTD}}

\subsection{Proof of Proposition~\ref{prop:RYDtoBKT}}

Fix $k<n$. We follow \cite{Prag}. The set $W^{OG(k,2n+1)}$ consists of all signed permutations of the form 
\[(y_1,y_2,\ldots , y_{k-r}, \overline{z_r}, \overline{z_{r-1}}, \ldots \overline{z_1}, v_1, v_2, \ldots v_{n-k})\]
where bars denote negative entries, $y_1<y_2<\ldots <y_{k-r}$, $z_r>z_{r-1}>\ldots > z_1$, $v_1<v_2<\ldots <v_{n-k}$ and $0\le r \le k$.

Define a {\bf PR shape} to be a pair of strict partitions $\alpha=(\alpha^{{\bf t}},\alpha^{{\bf b}})$ satisfying $\alpha^{{\bf t}}\subseteq (n-k)\times n$, $\alpha^{{\bf b}}\subseteq k\times n$ and $\alpha^{{\bf t}}_{n-k}\ge l(\alpha^{{\bf b}})+1$. Let $PR(k,n)$ denote the set of PR shapes. Then \cite{Prag} indexes the elements of $W^{OG(k,2n+1)}$ by PR shapes as follows:

\begin{Lemma}\label{lemma:PragBijection}\cite[Lemma 1.2]{Prag}
$W^{OG(k,2n+1)}$ is in bijection with $PR(k,n)$ via
\[\alpha^{{\bf b}}_j = n+1-z_j, \qquad 1\le j \le r\]
\[\alpha^{{\bf t}}_s = n+1-v_s+|\{q : z_q<v_s\}|, \qquad 1\le s \le n-k.\]
\end{Lemma}

Let $\alpha\in PR(k,n)$. Then $\tilde{\alpha}^{{\bf t}}:=\alpha^{{\bf t}} - (n-k,n-k-1,\ldots ,1)$ is a partition in $(n-k)\times k$.

Given $w\in W^{OG(k,2n+1)}$, let $Y=\{1,\ldots,k-r\}$, $Z = \{k-r+1,\ldots, k\}$ and $V=\{k+1,\ldots n\}$. Note that if $k+1-i\in Z$ then the $(k+1-i)$th entry of $w$ is $\overline{z_i}$, while if $k+1-i\in Y$ then the $(k+1-i)$th entry of $w$ is $y_{k+1-i}$. 

\begin{Claim}\label{claim:towardbijection}
For $1\le i\le k$, the length of the $i$th column of (the Ferrers diagram of) $\tilde{\alpha}^{{\bf t}}$ is $n-k$ if $k+1-i\in Z$, and $|\{l : y_{k+1-i}>v_l\}|$ if $k+1-i\in Y$.
\end{Claim}
\begin{proof}
By definition, the length of the $s$th row of $\tilde{\alpha}^{{\bf t}}$ is $k+s-v_s+|\{q : z_q<v_s\}| = k-|\{t:y_t<v_s\}|$. Then if $k+1-i\in Z$, the $i$th column has the maximal possible length $n-k$ since $k-|\{t:y_t<v_s\}|$ is never smaller than $k-|Y|$. Now suppose $k+1-i\in Y$. Then the length of the $i$th column is equal to the largest $s$ such that $y_{k+1-i}>v_s$, i.e., $|\{l : y_{k+1-i}>v_l\}|$.
\end{proof}

Let $(\tilde{\alpha}^{{\bf t}})'$ denote the conjugate partition of $\tilde{\alpha}^{{\bf t}}$. The bijection $PR(k,n)\rightarrow P(n-k,n)$ is given by $\alpha\mapsto (\tilde{\alpha}^{{\bf t}})'+\alpha^{{\bf b}}$ (see \cite[page 46]{BKT:Inventiones}.). 

\begin{Corollary}\label{cor:wordtoBKT}
$W^{OG(k,2n+1)}$ is in bijection with $P(n-k,n)$ via
\[\gamma_i=\begin{cases}
(n-k)+(n+1-z_{i}) & \text{if $k+1-i\in Z$}\\
|\{l : y_{k+1-i}>v_l\}| & \text{if $k+1-i\in Y$}.
\end{cases}\]
The Schubert variety indexed by $w\in W^{OG(k,2n+1)}$ is equal to the Schubert variety indexed by the image of $w$ in $P(n-k,n)$. 
\end{Corollary}
\begin{proof}
Compose the bijection $W^{OG(k,2n+1)}\rightarrow PR(k,n)$ of Lemma~\ref{lemma:PragBijection} with the bijection $PR(k,n)\rightarrow P(n-k,n)$, using Claim~\ref{claim:towardbijection}. 
\end{proof}

\begin{Example}\label{ex:BCindex}
Let $w=(2,3,7,\overline{8}, \overline{4},1,5,6)\in W^{OG(5,17)}$. The corresponding PR shape is $\alpha = ((8,5,4),(5,1))\in PR(5,8)$. Then $\tilde{\alpha}^{{\bf t}} = (5,3,3)$ and $(\tilde{\alpha}^{{\bf t}})'=(3,3,3,1,1)$. The corresponding $\gamma\in P(3,8)$ is $\gamma = (8,4,3,1,1)$. 
\end{Example}

Now we completely categorize the RYDs, similarly to the previous section, and give an explicit description of the RYD associated to a given $w\in W^{OG(k,2n+1)}$. In the standard embedding of the $B_n$ root system into $\mathbb{R}^n$, denote the root $e_a-e_b$ by $(a,b,-)$, $e_a+e_b$ by $(a,b,+)$, and $e_a$ by $(a)$. Then the base region consists of all $(a,b,\pm)$ with $a\ge k >b$ and all $(a)$ with $a\ge k$, while the top region consists of all $(a,b,+)$ with $a>b\ge k$. Let $w(a)$ denote the number in position $a$ of $w$, ignoring whether that entry is barred.

Call a subset $S\subset \Lambda_{k}$ a {\bf $W^{OG(k,2n+1)}$-diagram} if the roots in $S$ form a lower order ideal in each region, and also satisfy a {\bf support condition}: A root $(a,b,+)$ in the top region must be in $S$ if $S$ uses more than $2n+1-2k$ roots in the $a$th and $b$th rows combined, similarly, $(a,b)$ must not be in $S$ if $S$ uses fewer than $2n+1-2k$ roots in the $a$th and $b$th rows combined (compare this to the hook condition of the previous section). Let $\Theta(k,2n+1)$ denote the set of all $W^{OG(k,2n+1)}$-diagrams. The following lemma is proved by a straightforward computation of the inversion sets.

\begin{Lemma}\label{lemma:typeBRYDs}
$\mathbb{Y}_{OG(k,2n+1)}\subseteq \Theta(k,2n+1)$.
\end{Lemma}

\begin{Lemma}\label{lemma:RYDfromperm}
Let $w\in W^{OG(k,2n+1)}$ and let $\lambda\in \mathbb{Y}_{OG(k,2n+1)}$ be the corresponding RYD. Then
\[\lambda^{(1)}_i=\begin{cases}
n+1-k + |\{l : z_{i} < v_l\}| & \text{if $k+1-i\in Z$}\\
|\{l : y_{k+1-i} > v_l\}| & \text{if $k+1-i\in Y$}
\end{cases}\]
and 
\[\lambda^{(2)}_i=\begin{cases}
|\{q : z_{i} < z_q\}|+|\{t : z_{i} < y_t\}| & \text{if $k+1-i\in Z$}\\
0 & \text{if $k+1-i\in Y.$}
\end{cases}\]
\end{Lemma}
\begin{proof}

If $k+1-i \in Z$, then all $n-k$ roots of the form $(k+1-i,c,-)$, as well as $(k+1-i)$ in the base region are inverted by $w$. The roots of the form $(k+1-i,c,+)$ in the base inverted by $w$ are exactly those where $w(k+1-i)<w(c)$, so $\lambda_i^{(1)} = n+1-k + |\{l : z_{i} < v_l\}|$. If $k+1-i\in Y$, then neither $(k+1-i)$ nor any root of the form $(k+1-i,c,+)$ in the base is inverted by $w$. The roots in the base region of the form $(k+1-i,c,-)$ inverted by $w$ are those where $w(k+1-i)>w(c)$, so $\lambda_i^{(1)} = |\{l : y_{k+1-i} > v_l\}|$.

If $k+1-i\in Z$, then the roots of the top region of the form $(a,k+1-i,+)$ inverted by $w$ are those where $a\in Z$, or $a\in Y$ and $w(a)>w(k+1-i)$. Thus $\lambda^{(2)}_i= |\{q : z_{i} < z_q\}|+|\{t : z_{i} < y_t\}|$. If $k+1-i\in Y$, then the roots of the top region of the form $(a,k+1-i,+)$ have $a\in Y$ also, and no such roots can be inverted by $w$.
\end{proof}

\begin{Example}
Let $w=(2,3,7,\overline{8}, \overline{4},1,5,6)\in W^{(OG(5,17)}$, as in Example~\ref{ex:BCindex}. The corresponding RYD is $\lambda = ((6,4,3,1,1)|(2,0,0,0,0))\in \mathbb{Y}_{OG(5,17)}$.  
\end{Example}

\begin{Lemma}\label{lemma:injectiondiagramstoBKT}
The map $f_k$ of Proposition~\ref{prop:RYDtoBKT} is an injection $\Theta(k,2n+1)\rightarrow P(n-k,n)$.
\end{Lemma}
\begin{proof}
Let $\lambda\in\Theta(k,2n+1)$. It is clear from the definition of a $W^{OG(k,2n+1)}$-diagram that $f_k(\lambda)$ is a partition in $k\times (2n-k)$. To see that it is $(n-k)$-strict, suppose for some $i$ that $\lambda_i^{(1)} + \lambda_i^{(2)}>n-k$ and $\lambda_{i+1}^{(1)} + \lambda_{i+1}^{(2)}>n-k$. By the support condition, this implies $\lambda_i^{(1)}>n-k$ and $\lambda_{i+1}^{(1)}>n-k$. Then the support condition also implies that $\lambda_i^{(2)}>0$, since the root $(i,i+1,+)$ must be in $\lambda$. Then since $\lambda^{(2)}$ is strict, we have $\lambda_i^{(2)}>\lambda_{i+1}^{(2)}$. Thus $\lambda_i^{(1)} + \lambda_i^{(2)}>\lambda_{i+1}^{(1)} + \lambda_{i+1}^{(2)}$, and so $f_k(\lambda)\in P(n-k,n)$.

Now suppose for a contradiction that $f_k$ is not injective, i.e., there exist $\lambda,\mu \in \Theta(k,2n+1)$ such that $\lambda\neq \mu$ but $\lambda_i^{(1)} + \lambda_i^{(2)}=\mu_i^{(1)} + \mu_i^{(2)}$ for all $1\le i\le k$. Let $j$ largest such that $\lambda_j^{(1)}\neq\mu_j^{(1)}$ (such a $j$ must exist), and assume without loss of generality that $\lambda_j^{(1)}>\mu_j^{(1)}$. Then by the support condition, every root in the top region of the form $(a,j,+)$ which is in $\mu$ is also in $\lambda$. So $\lambda_j^{(2)}\ge\mu_j^{(2)}$, which contradicts the assumption that $\lambda_j^{(1)} + \lambda_j^{(2)}=\mu_j^{(1)} + \mu_j^{(2)}$.
\end{proof}

\begin{Corollary}\label{cor:Bbijection}
$\mathbb{Y}_{OG(k,2n+1)} = \Theta(k,2n+1)$. Furthermore, $f_k:\mathbb{Y}_{OG(k,2n+1)}\rightarrow P(n-k,n)$ is a bijection.
\end{Corollary}
\begin{proof}
Lemma~\ref{lemma:injectiondiagramstoBKT} gives an injection $\Theta(k,2n+1)\rightarrow P(n-k,n)$. Since Corollary~\ref{cor:wordtoBKT} establishes a bijection $ P(n-k,n)\rightarrow W^{OG(k,2n+1)}$, and by definition $W^{OG(k,2n+1)}$ is in bijection with $\mathbb{Y}_{OG(k,2n+1)}$, we have an injection $\Theta(k,2n+1)\rightarrow \mathbb{Y}_{OG(k,2n+1)}$. By Lemma~\ref{lemma:typeBRYDs}, $\mathbb{Y}_{OG(k,2n+1)} \subseteq \Theta(k,2n+1)$. Thus $\mathbb{Y}_{OG(k,2n+1)} = \Theta(k,2n+1)$, and the injection $f_k:\Theta(k,2n+1)\rightarrow P(n-k,n)$ is a bijection. 
\end{proof}

\noindent \emph{Proof of Proposition~\ref{prop:RYDtoBKT}:} By Corollary~\ref{cor:Bbijection}, we know $f_k:\mathbb{Y}_{OG(k,2n+1)}\rightarrow P(n-k,n)$ is a bijection. It remains to show $\lambda \in \mathbb{Y}_{OG(k,2n+1)}$ indexes the same Schubert variety as $f_k(\lambda)\in P(n-k,n)$.

Let $w\in W^{OG(k,2n+1)}$. Let $\lambda$ be the RYD indexing the same Schubert variety as $w$ by Lemma~\ref{lemma:RYDfromperm}, and let $\gamma$ be the element of $P(n-k,n)$ indexing the same Schubert variety as $w$ by Corollary~\ref{cor:wordtoBKT}. First suppose $k+1-i\in Z$. Then by Lemma~\ref{lemma:RYDfromperm}, $\lambda^{(1)}_i+\lambda^{(2)}_i = n+1-k + |\{l : z_{i} < v_l\}|+|\{q : z_{i} < z_q\}|+|\{t : z_{i} < y_t\}|$, which is equal to $n+1-k + (n- z_{i})$, which is equal to $\gamma_i$ by Corollary~\ref{cor:wordtoBKT}. Now suppose $k+1-i\in Y$. Then by Lemma~\ref{lemma:RYDfromperm}, $\lambda^{(1)}_i+\lambda^{(2)}_i = |\{l : y_{k+1-i} > v_l\}|$, which is equal to $\gamma_i$ by Corollary~\ref{cor:wordtoBKT}. Thus $\lambda$, $f_k(\lambda)$ index the same Schubert variety.
\qed

\subsection{Proof of Proposition~\ref{prop:RYDtoBKTD}}

Fix $k<n$. Using the same convention as in \cite{PragD}, the set $W^{OG(k,2n)}$ consists of all signed permutations that have an even number of signed entries, and are of the form 
\[(y_1,y_2,\ldots , y_{k-r}, \overline{z_r}, \overline{z_{r-1}}, \ldots \overline{z_1}, v_1, v_2, \ldots v_{n-k-1}, \widehat{v_{n-k}})\]
where $0\le r \le k$, bars denote negative entries, $y_1<y_2<\ldots <y_{k-r}$, $z_r>z_{r-1}>\ldots > z_1$, $v_1<v_2<\ldots <v_{n-k}$, and $\widehat{v_{n-k}}$ is either $v_{n-k}$ or $\overline{v_{n-k}}$, depending on the parity of $r$. Call $w$ a permutation of {\bf type I} if $\widehat{v_{n-k}}=v_{n-k}$, and {\bf type II} if $\widehat{v_{n-k}} = \overline{v_{n-k}}$. 

Given $w\in W^{OG(k,2n)}$, let $Y=\{1,\ldots,k-r\}$, $Z = \{k-r+1,\ldots, k\}$ and $V=\{k+1,\ldots n\}$. Note that if $k+1-i\in Z$ then the $(k+1-i)$th entry of $w$ is $\overline{z_i}$, while if $k+1-i\in Y$ then the $(k+1-i)$th entry of $w$ is $y_{k+1-i}$.

We now follow \cite{Tam:qcig}. Define a {\bf T-shape} to be a pair of partitions $\alpha=(\alpha^{{\bf t}},\alpha^{{\bf b}})$, where $\alpha^{{\bf b}} \subset k\times (n-1)$ is strict, $\alpha^{{\bf t}} \subset (n-k)\times k$, and $\alpha^{{\bf t}}_{n-k}\ge l(\alpha^{{\bf b}})$. Let $T(k,n)$ denote the set of all T-shapes.

The notation of \cite{Tam:qcig} differs from ours, specifically, the fork of the $D_n$ Dynkin diagram consists of nodes $1$ and $2$ in \cite{Tam:qcig} rather than $n-1$ and $n$. Translated into our notation, \cite{Tam:qcig} defines a surjection $h:W^{OG(k,2n)}\rightarrow T(k,n)$ via: 
\[\alpha^{{\bf t}}_i = k-v_i+i+ |\{j : z_j<v_i\}|\]
\[\alpha^{{\bf b}}_i = n- z_i\]
For $w\in W^{OG(k,2n)}$ such that $v_{n-k} = n$, $h$ is one-to-one. Otherwise $h$ is two-to-one, with 
\[(y_1,y_2,\ldots , y_{k-r}, \overline{n}, \overline{z_{r-1}}, \ldots \overline{z_1}, v_1, v_2, \ldots v_{n-k-1}, \widehat{v_{n-k}})\mbox{ \ and \ }\] 
\[(y_1,y_2,\ldots , y_{k-r}, n, \overline{z_{r-1}}, \ldots \overline{z_1}, v_1, v_2, \ldots v_{n-k-1}, \widehat{v_{n-k}})\] 
mapping to the same T-shape. One of these permutations has type I, the other has type II. 

Let $T'(k,n)$ be the set containing a single copy of each $\alpha\in T(k,n)$ that satisfies $|h^{-1}(\alpha)|=1$, and two copies of each $\alpha\in T(k,n)$ that satisfies $|h^{-1}(\alpha)|=2$, where one copy is declared to have type 1 and the other copy type 2. Define a map $h':W^{OG(k,2n)}\rightarrow T'(k,n)$ by letting $h'(w) = h(w)$ whenever $h$ is one-to-one, and whenever $h$ is two-to-one let $h'(w)$ be the T-shape $h(w)$ of type 1 (respectively, type 2) if $w$ is of type I (respectively, type II). Then $h'$ is a bijection. Note that the definition of \emph{type} of a T-shape used here is not the same as that used by \cite{Tam:qcig}.  

\begin{Claim}\label{claim:Dtowardbijection}
Let $w\in W^{OG(k,2n)}$ and let $h(w)=\alpha$ be the corresponding T-shape. Then for $1\le i\le k$, the length of the $i$th column of $\alpha^{{\bf t}}$ is $n-k$ if $k+1-i\in Z$, and $|\{l : y_{k+1-i}>v_l\}|$ if $k+1-i\in Y$.
\end{Claim}
\begin{proof}
Identical to the proof of Claim~\ref{claim:towardbijection}.
\end{proof}

Given $\alpha\in T(k,n)$, let $(\alpha^{{\bf t}})'$ denote the conjugate partition of $\alpha^{{\bf t}}$. Now we follow \cite[pp 46--47]{BKT:Inventiones}. The bijection $T'(k,n)\rightarrow \tilde{P}(n-k,n)$ is given by $\alpha\mapsto (\alpha^{{\bf t}})'+ \alpha^{{\bf b}}$, where if $\alpha$ has type 1 (respectively, 2), its image in $\tilde{P}(n-k,n)$ has type 1 (respectively, 2). 

\begin{Corollary}\label{cor:DwordtoBKT}
$W^{OG(k,2n)}$ is in bijection with $\tilde{P}(n-k,n)$ via
\[\gamma_i=\begin{cases}
(n-k)+(n-z_{i}) & \text{if $k+1-i\in Z$}\\
|\{l : y_{k+1-i}>v_l\}| & \text{if $k+1-i\in Y$}.
\end{cases}\]
where $\tilde{\gamma}=(\gamma;0)$ if $\gamma$ has no part of size $n-k$, otherwise $\tilde{\gamma}=(\gamma;1)$ if $w$ has type I and $\tilde{\gamma}=(\gamma;2)$ if $w$ has type II. The Schubert variety indexed by $w\in W^{OG(k,2n)}$ is equal to the Schubert variety indexed by the image of $w$ in $\tilde{P}(n-k,n)$.
\end{Corollary}
\begin{proof}
Compose the bijection $W^{OG(k,2n)}\rightarrow T'(k,n)$ with the bijection $T'(k,n)\rightarrow \tilde{P}(n-k,n)$, using Claim~\ref{claim:Dtowardbijection}.
It is clear that $\gamma$ has a part of size $n-k$ if and only if either $z_r=n$ or $y_{k-r}=n$ in $w$. 
\end{proof}

\begin{Example}\label{ex:Dindex}
Let $w=(2,4,\overline{8},\overline{6},\overline{1},3,5,\overline{7})\in W^{OG(5,16)}$. The corresponding T-shape is $\alpha = ((4,3,3),(7,2,0))$ (type $2$). Then $(\alpha^{{\bf t}})'=(3,3,3,1,0)$. The corresponding $\tilde{\gamma}\in \tilde{P}(3,8)$ is $\tilde{\gamma}=((10,5,3,1,0);2)$.  
\end{Example}

Now we completely categorize the RYDs, and give an explicit description of the RYD associated to $w\in W^{OG(k,2n)}$. In the standard embedding of the $D_n$ root system into $\mathbb{R}^n$, denote the root $e_a-e_b$ by $(a,b,-)$ and $e_a+e_b$ by $(a,b,+)$. Call a subset $S\subset \Lambda_{k}$ a {\bf $W^{OG(k,2n)}$-diagram} if the roots in $S$ form a lower order ideal in each region, and also satisfy a {\bf support condition} similar to that of type B/C: a root $(a,b,+)$ in the top region must be in $S$ if $S$ uses more than $2n-2k$ roots from the $a$th and $b$th double-tailed diamonds, similarly, $(a,b,+)$ must not be in $S$ if $S$ uses fewer than $2n-2k$ roots from the $a$th and $b$th double-tailed diamonds. Let $\Theta(k,2n)$ denote the set of all $W^{OG(k,2n)}$-diagrams.
The following lemma is proved by a straightforward computation of the inversion sets.

\begin{Lemma}\label{lemma:typeDRYDs}
$\mathbb{Y}_{OG(k,2n)}\subseteq \Theta(k,2n)$.
\end{Lemma}

\begin{Lemma}\label{lemma:DRYDfromperm}
Let $w\in W^{OG(k,2n)}$ and let $\lambda\in \mathbb{Y}_{OG(k,2n)}$ be the corresponding RYD. Then
\[\lambda^{(1)}_i=\begin{cases}
n-k + |\{l : z_{i} < v_l\}| & \text{if $k+1-i\in Z$}\\
|\{l : y_{k+1-i} > v_l\}| & \text{if $k+1-i\in Y,$}
\end{cases}\]

\[\lambda^{(2)}_i=\begin{cases}
|\{q : z_{i} < z_q\}|+|\{t : z_{i} < y_t\}| & \text{if $k+1-i\in Z$}\\
0 & \text{if $k+1-i\in Y$}
\end{cases}\]
and if $\lambda_i^{(1)}=n-k$ roots for some $i$, then $\lambda$ is assigned $\uparrow$ if $w$ is of type I and $\downarrow$ if $w$ is of type II. 
\end{Lemma}
\begin{proof}

\noindent\emph{($w$ is of type I):} If $k+1-i \in Z$, then all $n-k$ roots $(k+1-i,c,-)$ in the base are inverted by $w$. The roots of the form $(k+1-i,c,+)$ in the base inverted by $w$ are exactly those where $w(k+1-i)<w(c)$, so $\lambda_i^{(1)} = n-k + |\{l : z_{i} < v_l\}|$. If $k+1-i\in Y$, then no roots of the form $(k+1-i,c,+)$ in the base are inverted by $w$. The roots in the base of the form $(k+1-i,c,-)$ inverted by $w$ are those where $w(k+1-i)>w(c)$, so $\lambda_i^{(1)} = |\{l : y_{k+1-i} > v_l\}|$.

If $k+1-i\in Z$, then the roots of the top region of the form $(a,k+1-i,+)$ inverted by $w$ are those where either $a\in Z$, or $a\in Y$ and $w(a)>w(k+1-i)$. Thus $\lambda^{(2)}_i= |\{q : z_{i} < z_q\}|+|\{t : z_{i} < y_t\}|$. If $k+1-i\in Y$, then the roots of the top region of the form $(a,k+1-i,+)$ have $a\in Y$ also, and no such roots can be inverted by $w$.

\noindent\emph{($w$ is of type II):} If $k+1-i \in Z$, then all $n-k-1$ roots $(k+1-i,c,-)$ for $c<n$ in the base are inverted by $w$, and also $(k+1-i,n,+)$ is inverted by $w$. The number of remaining roots of the $i$th double-tailed diamond inverted by $w$ is
\[|\{l<n-k: z_{i}<v_l\}| + 
\begin{cases}
1 & \text{if $z_{i}< v_{n-k}$}\\
0 & \text{if $z_{i}> v_{n-k}$} 
\end{cases} \]
(the first summand is the number of $(k+1-i,c,+)$ for $c<n$ inverted, the second is whether $(k+1-i,n,-)$ is inverted). Thus $\lambda_i^{(1)} = n-k + |\{l : z_{i} < v_l\}|$. If $k+1-i\in Y$, then no roots of the form $(k+1-i,c,+)$ for $c<n$ in the base are inverted by $w$, and also $(k+1-i,n,-)$ is not inverted by $w$. Thus the number of roots of the $i$th double-tailed diamond inverted by $w$ is
\[|\{l<n-k : y_{k+1-i} > v_l\}| + 
\begin{cases}
1 & \text{if $y_{k+1-i}> v_{n-k}$}\\
0 & \text{if $y_{k+1-i}< v_{n-k}$} 
\end{cases} \]
(the first summand is the number of $(k+1-i,c,-)$ for $c<n$ inverted, the second is whether $(k+1-i,n,+)$ is inverted). Thus $\lambda_i^{(1)} = |\{l : y_{k+1-i} > v_l\}|$.

Since the last co-ordinate of any root of the top region is zero, it is irrelevant whether the last entry of $w$ is barred. Hence for $\lambda_i^{(2)}$, the statement for the top region follows by the same argument as for type I permutations.

Finally, if $\lambda_i^{(1)}=n-k$ for some $i$, then $\lambda$ uses either $(k+1-i,n,-)$ (above $\beta_{n-1}$) or $(k+1-i,n,+)$ (above $\beta_{n}$) but not both. If $\lambda$ uses the former but not the latter then the last entry of $w$ must be unbarred (i.e., $w$ is of type I), and if it uses the latter but not the former then similarly $w$ must be of type II. Thus $\lambda$ is assigned $\uparrow$ (respectively, $\downarrow$) if and only if $\lambda_i^{(1)}=n-k$ for some $i$ and $w$ is of type I (respectively, type II). 
\end{proof}

\begin{Example}
Let $w=(2,4,\overline{8},\overline{6},\overline{1},3,5,\overline{7})\in W^{OG(5,16)}$, as in Example~\ref{ex:Dindex}. The corresponding RYD is $\lambda = ((6,4,3,1,0)|(4,1,0,0,0))^\downarrow\in \mathbb{Y}_{OG(5,16)}$. 
\end{Example}

\begin{Lemma}\label{lemma:DinjectiondiagramstoBKT}
The map $F_k$ of Proposition~\ref{prop:RYDtoBKTD} is an injection $\Theta(k,2n)\rightarrow \tilde{P}(n-k,n)$.
\end{Lemma}
\begin{proof}
Let $\lambda\in \Theta(k,2n)$. It is clear from the definition of a $W^{OG(k,2n)}$-diagram that for $\tilde{\gamma}=F_k(\lambda)$, $\gamma$ is a partition in $k\times (2n-1-k)$. First we show $\gamma$ is $(n-k)$-strict. Suppose for some $i$ that $\lambda_i^{(1)} + \lambda_i^{(2)}>n-k$ and $\lambda_{i+1}^{(1)} + \lambda_{i+1}^{(2)}>n-k$. By the support condition, this implies $\lambda_i^{(1)}\ge n-k$ and $\lambda_{i+1}^{(1)}\ge n-k$. If the first inequality is strict then the support condition also implies that $\lambda_i^{(2)}>0$ since the root $(i,i+1,+)$ must be in $\lambda$, while if it is an equality then we also have $\lambda_i^{(2)}>0$ since $\lambda_i^{(1)} + \lambda_i^{(2)}>n-k$. Since $\lambda^{(2)}$ is a strict partition, this implies $\lambda_i^{(2)}>\lambda_{i+1}^{(2)}$, whence $\lambda_i^{(1)} + \lambda_i^{(2)}>\lambda_{i+1}^{(1)} + \lambda_{i+1}^{(2)}$.

Next, to demonstrate that $F_k$ is well-defined, we show that $\lambda^{(1)}$ has a row of length $n-k$ if and only if $\gamma$ has a row of length $n-k$. Suppose $\lambda^{(1)}$ has a row of length $n-k$, and let $i$ be largest such that $\lambda^{(1)}_i=n-k$. Then $\lambda^{(1)}_{l}<n-k$ for all $l>i$, and thus by the support condition $\lambda^{(2)}_i= 0$. So $\gamma_i=n-k$.
Now suppose $\lambda^{(1)}$ has no row of length $n-k$, and consider an arbitrary row $\lambda^{(1)}_i$ of $\lambda^{(1)}$. If $\lambda^{(1)}_i>n-k$ then clearly $\gamma_i>n-k$. If $\lambda^{(1)}_i<n-k$ then $\lambda^{(1)}_l<n-k$ for all $l>i$, and then by the support condition $\lambda^{(2)}_i=0$. Hence $\gamma_i = \lambda^{(1)}_i<n-k$.

The argument that $F_k$ is injective is then similar to that of Lemma~\ref{lemma:injectiondiagramstoBKT}.
\end{proof}

\begin{Corollary}\label{cor:Dbijection}
$\mathbb{Y}_{OG(k,2n)}=\Theta(k,2n)$. Furthermore, $F_k:\mathbb{Y}_{OG(k,2n)}\rightarrow \tilde{P}(n-k,n)$ is a bijection.
\end{Corollary}
\begin{proof}
Identical to the proof of Corollary~\ref{cor:Bbijection}, using instead Lemmas~\ref{lemma:typeDRYDs}, \ref{lemma:DinjectiondiagramstoBKT} and Corollary~\ref{cor:DwordtoBKT}.
\end{proof}

\noindent \emph{Proof of Proposition~\ref{prop:RYDtoBKTD}:} By Corollary~\ref{cor:Dbijection}, we know $F_k$ is a bijection $\mathbb{Y}_{OG(k,2n)}\rightarrow \tilde{P}(n-k,n)$.
It remains to show the image of $\lambda$ indexes the same Schubert variety as $\lambda$.

Let $w\in W^{OG(k,2n)}$. Let $\lambda$ be the RYD indexing the same Schubert variety as $w$ by Lemma~\ref{lemma:DRYDfromperm}, and let $\tilde{\gamma}=(\gamma;{\tt type}(\gamma))$ be the element of $\tilde{P}(n-k,n)$ indexing the same Schubert variety as $w$ by Corollary~\ref{cor:DwordtoBKT}. First suppose $k+1-i\in Z$. Then by Lemma~\ref{lemma:DRYDfromperm}, $\lambda^{(1)}_i+\lambda^{(2)}_i = n-k + |\{l : z_{i} < v_l\}|+|\{q : z_{i} < z_q\}|+|\{t : z_{i} < y_t\}|$, which is equal to $n-k + (n- z_{i})$, which is equal to $\gamma_i$ by Corollary~\ref{cor:DwordtoBKT}. Now suppose $k+1-i\in Y$. By Lemma~\ref{lemma:DRYDfromperm}, $\lambda^{(1)}_i+\lambda^{(2)}_i =|\{l : y_{k+1-i} > v_l\}|$, which is equal to $\gamma_i$ by Corollary~\ref{cor:DwordtoBKT}.

By the proof of Lemma~\ref{lemma:DinjectiondiagramstoBKT}, either $\lambda^{(1)}, \gamma$ both have a row of length $n-k$ or both do not. If they do, then if $w$ is of type I, $\lambda$ is assigned $\uparrow$ and $\gamma$ is of type 1, while if $w$ is of type II, $\lambda$ is assigned $\downarrow$ and $\gamma$ is of type 2. Thus $\lambda$, $F_k(\lambda)$ index the same Schubert variety.
\qed

\section{Proof of Theorem~\ref{Thm:Pieriagreement}(I)}

We follow \cite[pg. 3-5]{BKT:Inventiones}. The Schubert varieties of $LG(2,2n)$ are indexed by the set $P(n-2,n)$ of $(n-2)$-strict partitions inside a $2\times (2n-2)$ rectangle. The {\bf Pieri classes} of \cite{BKT:Inventiones} are those indexed by $\gamma=(p,0)\in P(n-2,n)$. Denote these classes by $\sigma_p$. 

Fix an integer $p\in [1,2n-2]$, and suppose $\gamma, \delta\in P(n-2,n)$ with $|\delta|=|\gamma|+p$. Call a box of $\delta$ a $\delta$-box, a box of $\gamma$ a $\gamma$-box, a box of $\delta$ that is not in $\gamma$ a $(\delta\setminus\gamma)$-box, and  a box of $\gamma$ that is not in $\delta$ a $(\gamma\setminus\delta)$-box. We say the box in row $r$ and column $c$ of $\gamma$ is {\bf related} to the box in row $r'$ and column $c'$ if $|c-(n-1)|+r = |c'-(n-1)|+r'$. Then there is a relation $\gamma \rightarrow \delta$ if $\delta$ can be obtained by removing a vertical strip from the first $n-2$ columns of $\gamma$ and adding a horizontal strip to the result, such that
\begin{enumerate}
\item Each $\gamma$-box in the first $n-2$ columns having no $\delta$-box below it is related to at most one $(\delta\setminus\gamma)$-box.
\item Any $(\gamma\setminus\delta)$-box and the box above it must each be related to exactly one $(\delta \setminus \gamma)$-box, and these $(\delta \setminus \gamma)$-boxes must all lie in the same row.
\end{enumerate}
If $\gamma \rightarrow \delta$, let $\mathbb{A}$ be the set of $(\delta \setminus \gamma)$-boxes in columns $n-1$ through $2n-2$ which are \emph{not} mentioned in (1) or (2). Define two boxes of $\mathbb{A}$ to be {\bf connected} if they share at least a vertex. Then define $N(\gamma,\delta)$ to be the number of connected components of $\mathbb{A}$ that do not use a box of the $(n-1)$th column. 

Then the specialization of the Pieri rule of \cite[Theorem 1.1]{BKT:Inventiones} to the coadjoint $LG(2,2n)$ is

\begin{Theorem}[\cite{BKT:Inventiones}]\label{TheoremBKTPieriC}(Pieri rule for LG(2,2n))
For any $\gamma \in P(n-2,n)$ and integer $p \in [1,2n-2]$,
\begin{equation}\nonumber
\sigma_p \cdot \sigma_{\gamma} = \sum_{\delta} 2^{N(\gamma,\delta)}\sigma_{\delta}
\end{equation}
where the sum is over all $\delta \in P(n-2,n)$ with $\gamma \rightarrow \delta$. 
\end{Theorem}

Let $(r:c)$ denote the box in row $r$, column $c$ of $2\times (2n-2)$. Let $L$ denote the first $n-2$ columns of $2\times(2n-2)$ and $R$ the latter $n$ columns. Given $\gamma,\delta\in P(n-2,n)$ with $|\delta|=|\gamma|+p$, let $\mathbb{D}_1$ denote the set of $(\delta\setminus\gamma)$-boxes in row 1 of $R$, and $\mathbb{D}_2$ the set of $(\delta\setminus\gamma)$-boxes in row 2 of $R$. Let $\mathbb{D}=\mathbb{D}_1\cup\mathbb{D}_2$. By definition, both $\mathbb{D}_1$, $\mathbb{D}_2$ are connected and

\begin{Lemma}\label{lemma:mathbbDC}
$\mathbb{D}_1 = \begin{cases}
\{(1:c) : \gamma_1+1\le c \le \delta_1\} & \text{if $\gamma_1>n-2$} \\
\{(1:c) : n-1\le c\le \delta_1\} & \text{if $\gamma_1\le n-2$}
\end{cases}$ \qquad and \qquad

$\mathbb{D}_2=\begin{cases}
\{(2:c) : \gamma_2+1\le c\le \delta_2\} & \text{if $\gamma_2>n-2$} \\
\{(2:c) : n-1\le c \le \delta_2\} & \text{if $\gamma_2\le n-2$.}
\end{cases}$
\end{Lemma}

Let $\gamma^*$ denote the shape $(\gamma_1+p+1,\gamma_2-1)$. We gather some facts about which pairs $\gamma,\delta$ satisfy $\gamma\rightarrow\delta$. 

\begin{Lemma}\label{lemma:removeboxC}
If $\gamma \rightarrow \delta$ and $\gamma \not\subseteq \delta$, then $\delta = \gamma^*$.
\end{Lemma}
\begin{proof}
Boxes removed from $\gamma$ must be a vertical strip, so at most one box can be removed from each row of $\gamma$. Since a horizontal strip of boxes must be added after removing the vertical strip, we may assume boxes are not removed from both rows and that all $(\delta\setminus\gamma)$-boxes are added in the row from which we did not remove a box. The claim follows by noting $(\gamma_1-1,\gamma_2+p+1)$ is either not a partition or has no boxes in the last $n$ columns, violating (2). 
\end{proof}

\begin{Lemma}\label{lemma:threshC}
Suppose $|\gamma|\le 2n-3$ and $p+|\gamma|>2n-3$. If $\gamma^*\in P(n-2,n)$, then $\gamma \rightarrow \gamma^*$.
\end{Lemma}
\begin{proof}
Let $\delta=\gamma^*$. All $\mathbb{D}$-boxes are in row 1, thus (1) holds. The $(\gamma\setminus\delta)$-box $(2:\delta_2+1)$ is related to $(1:2n-2-\delta_2)$ and the box $(1:\gamma_2)$ above $(2:\delta_2+1)$ is related to $(1: 2n-2-\gamma_2)$. Since $\gamma_1+1\le 2n-2-\gamma_2<2n-2-\delta_2\le \delta_1$, we have $(1:2n-2-\delta_2)$ and $(1:2n-2-\gamma_2)$ are different $\mathbb{D}$-boxes. Hence (2) holds. 
\end{proof}

\begin{Lemma}\label{lemma:nonthreshC}
If either $|\delta|\le 2n-3$ or $|\gamma|>2n-3$, then $\gamma \rightarrow \delta \Rightarrow \gamma \subseteq \delta$. In particular, $\delta$ is obtained from $\gamma$ without removing any box of $\gamma$.
\end{Lemma}
\begin{proof}
Assume for a contradiction that $\gamma \rightarrow \delta$ but $\gamma \not\subseteq \delta$. Then by Lemma~\ref{lemma:removeboxC}, $\delta = \gamma^*$. Suppose $|\gamma|>2n-3$. Then the box $(1:\gamma_2)$ above the removed box is related to $(1:2n-2-\gamma_2)$, which is not in $\mathbb{D}$ since $\gamma_1+1>2n-2-\gamma_2$. This violates (2). Suppose $|\delta|\le 2n-3$. Then the removed box $(2:\delta_2+1)$ is related to $(1:2n-2-\delta_2)$, which is not in $\mathbb{D}$ since $\delta_1<2n-2-\delta_2$. This violates (2).
\end{proof}

Given $\gamma\rightarrow\delta$, we will say a box of $\mathbb{D}$ is {\bf killed} if it is mentioned in (1) or (2), i.e., if it is not in $\mathbb{A}$. We will say a connected component $D$ of $\mathbb{D}$ is {\bf bisected} if a box $\mathfrak{d}$ of $D$ is killed but there exist boxes of $D$ in both earlier and later columns than $\mathfrak{d}$, which are not killed. The following lemmas will help us in computing $N(\gamma,\delta)$.

\begin{Lemma}\label{lemma:specialPiericaseC}
If $\gamma^*\in P(n-2,n)$ and $\gamma\rightarrow\gamma^*$, then $N(\gamma,\delta)=0$.
\end{Lemma}
\begin{proof}
Let $\delta=\gamma^*$. If $\gamma_1\ge n-2$, all boxes of $R$ except $(1:n-1)$ are mentioned in (1) or (2), so $N(\gamma,\delta)=0$. Suppose $\gamma_1<n-2$. Then $\mathbb{D}_2=\emptyset$, so $\mathbb{D}=\mathbb{D}_1$. By (1), (2) it is clear the $\mathbb{D}_1$-boxes killed are the last $l$ boxes of $\mathbb{D}_1$ for some $l>0$, hence $\mathbb{D}_1$ is not bisected. Thus $\mathbb{A}$ is a single component containing $(1:n-1)$, whence $N(\gamma,\delta)=0$.
\end{proof}

Whenever $\gamma\rightarrow\delta$ with $\gamma\subset\delta$, define 
\[S=\{(1:c) : \delta_2+1\le c \le \gamma_1\}\cap L\ \qquad \mbox{and} \qquad T= \{(2:c) : 1\le c \le \gamma_2)\}\cap L.\]  
By definition, the boxes of $S$ and $T$ are the $\gamma$-boxes considered in (1), hence the only boxes capable of killing $\mathbb{D}$-boxes.

\begin{Lemma}\label{lemma:killedC} 
Let $\gamma\rightarrow\delta$ with $\gamma\subset\delta$. Suppose $(1:c)\in\mathbb{D}_1$. If $c=n-1$ then $(1:c)$ is not killed, while if $c\neq n-1$ then
\begin{itemize}
\item $(1:c)$ is killed by $S$ if and only if $(1:c)\in S'_1 = \{(1:c') : 2n-2-\gamma_1\le c'\le 2n-3-\delta_2\}$
\item $(1:c)$ is killed by $T$ if and only if $(1:c) \in T'_1 = \{(1:c') : 2n-1-\gamma_2\le c'\le 2n-2\}.$
\end{itemize}
Suppose $(2:c)\in\mathbb{D}_2$. If $c=n-1$ then $(2:c)$ is not killed, while if $c\neq n-1$ then
\begin{itemize}
\item $(2:c)$ is never killed by $S$
\item $(2:c)$ is killed by $T$ if and only if $(2:c) \in T'_2 = \{(2:c') : 2n-2-\gamma_2\le c'\le 2n-3\}.$
\end{itemize}
\end{Lemma}
\begin{proof}
Clearly $(1:n-1)$, $(2:n-2)$ can never be killed. The existence of a $\mathbb{D}$-box in row 2 implies $\delta_2>n-2$ and thus $S=\emptyset$, so $(2:c)$ is never killed by $S$ and also $(2:n-1)$ can never be killed. The remaining points also follow from the definition of being related.
\end{proof}

\begin{Corollary}\label{cor:killedC}
Suppose $\gamma\rightarrow\delta$ with $\gamma\subset\delta$. Then if $(1:2n-2-\delta_2)$ is a $\mathbb{D}_1$-box, it is not killed.
\end{Corollary}
\begin{proof}
Since $2n-3-\delta_2<2n-2-\delta_2<2n-1-\gamma_2$, $(1:2n-2-\delta_2)$ is not in $S'_1$ or $T'_1$.
\end{proof}

\begin{Lemma}\label{lemma:bisectionC}
A connected component of $\mathbb{D}$ is bisected if and only if all of the following hold:
\begin{itemize}
\item[(i)] $|\gamma|\le 2n-3$ and $|\delta|>2n-3$
\item[(ii)] $\gamma\subseteq\delta$
\item[(iii)] $\gamma_1<n-1$
\item[(iv)] $\delta_2<\gamma_1.$
\end{itemize} 
\end{Lemma}
\begin{proof}
($\Rightarrow$, by contrapositive) If (ii) does not hold, then by the proof of Lemma~\ref{lemma:specialPiericaseC} no component of $\mathbb{D}$ is bisected, so assume (ii) holds. Then for a given component $D$ of $\mathbb{D}$, by Lemma~\ref{lemma:killedC} $T$ kills the latest $l$ boxes of $D$ for some $l\ge 0$ and thus does not bisect $D$. So only $S$ can bisect $D$. If (iv) does not hold, then $S=\emptyset$ and $\mathbb{D}$ cannot be bisected. Suppose (iii) does not hold. We may assume $\mathbb{D}_2=\emptyset$, otherwise $S=\emptyset$ and we are done. Then $\mathbb{D}=\mathbb{D}_1$, and since $2n-2-\gamma_1\le \gamma_1+1$, we have $\mathbb{D}_1\setminus S'_1$ is connected. Finally, suppose (i) does not hold. Then either $|\gamma|> 2n-3$ or $|\delta|\le 2n-3$. We may assume the latter three conditions hold. Then (iii) implies $|\gamma|<2n-3$, so we must have $|\delta|\le 2n-3$. Then $\mathbb{D}=\mathbb{D}_1$. Since $2n-3-\delta_2\ge \delta_1$, $\mathbb{D}_1\setminus S'_1$ is connected.

($\Leftarrow$) Suppose all four conditions hold. Then by (iii) and (iv), $\delta_2<n-2$, so $\mathbb{D}=\mathbb{D}_1$. By (i) $|\delta|>2n-3$, so $\delta_1>n-1$, and since by (iii) $\gamma_1<n-1$, we have $(1:n-1)$ is a $\mathbb{D}_1$-box and is not killed. Next, $(1:2n-2-\delta_2)$ is a $\mathbb{D}_1$-box since by (i) $2n-2-\delta_2\le \delta_1$, and by Corollary~\ref{cor:killedC} it is not killed. Finally, since by (iv) $\delta_2<\gamma_1$ we have $n-1<2n-2-\gamma_1\le 2n-3-\delta_2<2n-2-\delta_2$. In particular, $S'_1\neq\emptyset$, so a $\mathbb{D}_1$-box between $(1:n-1)$ and $(1:2n-2-\delta_2)$ is killed. Hence $\mathbb{D}_1$ is bisected.
\end{proof}

\begin{Corollary}\label{cor:bisectionC}
If a connected component of $\mathbb{D}$ is bisected, then $N(\gamma,\delta)=1$.
\end{Corollary}
\begin{proof}
By the proof of Lemma~\ref{lemma:bisectionC}, if a connected component of $\mathbb{D}$ is bisected then $\mathbb{D}=\mathbb{D}_1$, so $\mathbb{D}_1$ is bisected. It also follows from the proof that $\mathbb{D}_1\setminus(S'_1\cup T'_1)=\mathbb{A}$ has two connected components, one of which uses $(1:n-1)$. Thus $N(\gamma,\delta)=1$.
\end{proof}

\begin{Lemma}\label{lemma:threshcptsurvivesC}
If $\gamma\rightarrow\delta$ with $\gamma\subset\delta$, $|\gamma|\le 2n-3$, $|\delta|>2n-3$, $\gamma_1\ge n-1$ and also $\mathbb{D}_1$ is nonempty, then not all $\mathbb{D}_1$-boxes are killed.
\end{Lemma}
\begin{proof}
Since $\delta_1>2n-3-\delta_2$, we have $(1:\delta_1)\in \mathbb{D}_1\setminus S'_1$. Since $\gamma_1+1<2n-1-\gamma_2$, we have $(1:\gamma_1+1)\in \mathbb{D}_1\setminus T'_1$. Thus if either $S'_1$ or $T'_1$ is empty, we are done. If both $S'_1$ and $T'_1$ are nonempty, then $(2n-2-\delta_2)$ is a $\mathbb{D}_1$-box since $2n-3-\delta_2<2n-2-\delta_2<2n-1-\gamma_2$. By Corollary~\ref{cor:killedC} it is not killed.
\end{proof}

Now we consider the RYD model. In the coadjoint case $k=2$, the base region is a $2\times (2n-3)$ rectangle and the top region is a single root. From now on, we will use the notation of \cite{Searles.Yong} for the RYDs. An RYD for $LG(2,2n)$ will be denoted $\olambda = \langle \lambda|\bullet\rangle$ or $\olambda = \langle \lambda|\circ\rangle$ where $\lambda=(\lambda_1,\lambda_2)$ is the partition in $2\times (2n-3)$ corresponding to the roots used in the base region, and $\bullet/\circ$ denotes whether $\olambda$ uses the single root in the top region or not. We will denote the set of RYDs for $LG(2,2n)$ by ${\mathbb Y}_{LG(2,2n)}$ (this set is the same as ${\mathbb Y}_{OG(2,2n+1)}$ from the introduction). Let $\olambda, \omu\in \mathbb{Y}_{LG(2,2n)}$, and let $M=\min\{\lambda_1-\lambda_2, \mu_1-\mu_2\}$. We reprise the definition of the product $\star$ on RYDs from \cite[Theorem 4.1]{Searles.Yong}: 

\begin{Definition}\cite{Searles.Yong}\label{def:LGproduct}
Define a commutative product $\star$ on $\mathbb{Z}[\mathbb{Y}_{LG(2,2n)}]$:
\begin{enumerate}
\item[(A)] If $|\langle\lambda|\circ\rangle| + |\langle \mu|\circ\rangle| \le 2n-3$, then
\[\langle\lambda|\circ\rangle \star \langle \mu|\circ\rangle = \sum_{0 \le k \le M} \langle\lambda_1+\mu_1-k, \lambda_2+\mu_2+k| \circ\rangle\]
\item[(B)] If $|\langle\lambda|\circ\rangle| + |\langle \mu|\circ\rangle| > 2n-3$, then
\[\langle\lambda|\circ\rangle\star \langle \mu|\circ\rangle =
\sum_{0 \le k \le M} [\langle\lambda_1+\mu_1-k, \lambda_2+\mu_2+k-1| \bullet\rangle+ \langle\lambda_1+\mu_1-k-1, \lambda_2+\mu_2+k| \bullet\rangle]\]
\item[(C)] \[\langle\lambda|\bullet\rangle\star \langle \mu|\circ\rangle = \langle\lambda|\circ\rangle\star \langle \mu|\bullet\rangle =
\sum_{0 \le k \le M} \langle\lambda_1+\mu_1-k, \lambda_2+\mu_2+k| \bullet\rangle\]
\item[(D)] $\langle\lambda|\bullet\rangle \star \langle \mu|\bullet\rangle = 0$.
\end{enumerate}
Declare any $\overline{\alpha}$ in the above expressions to be zero if
$(\alpha_1,\alpha_2)$ is not a partition in $2\times (2n-3)$. Such $\overline{\alpha}$ will be called {\bf illegal}.
\end{Definition}

The following specializes Proposition~\ref{prop:RYDtoBKT} to the case $k=2$. We write $f$ instead of $f_2$.

\begin{Proposition} The elements of ${\mathbb Y}_{LG(2,2n)}$ are in bijection with the elements of $P(n-2,n)$ via

$f(\olambda)=\begin{cases}
(\lambda_1,\lambda_2) & \text{if $\olambda=\langle \lambda|\circ\rangle$}\\
(\lambda_1+1,\lambda_2) & \text{if $\olambda=\langle \lambda|\bullet\rangle$}
\end{cases}$
\end{Proposition}

Let $\oalpha_p$ denote $\langle p,0|{\bullet}/{\circ}\rangle\in \mathbb{Y}_{LG(2,2n)}$, and given $\olambda\in\mathbb{Y}_{LG(2,2n)}$ let $\gamma$ denote $f(\olambda)$. 

\begin{Lemma}\label{lemma:shapesagreeC}
Suppose $p\neq 2n-2$. Then a (legal) shape $\omu$ appears in the expansion $\oalpha_p\star\olambda$ if and only if $f(\omu)$ appears in the expansion $\sigma_p\cdot\sigma_{\gamma}$.
\end{Lemma}
\begin{proof}
Let $\Delta=\{\delta\in P(n-2,n) : \gamma\subset\delta \text{\ and\ } |\delta|=|\gamma|+p\}$. There are three cases:

($p+|\olambda|\le 2n-3$:) By (A), the shapes in $\oalpha_p\star\olambda$ are those created by adding a horizontal strip of size $p$ to $\lambda$. The image of the legal shapes under $f$ are $\Delta$. Every element of $\Delta$ satisfies (1) and (2), so $\gamma\rightarrow\delta$ for every element $\delta$ of $\Delta$. By Lemma~\ref{lemma:nonthreshC}, there are no other $\delta'\in P(n-2,n)$ such that $\gamma\rightarrow\delta'$.

($|\olambda|>2n-3$:) By (C), the shapes in $\oalpha_p\star\olambda$ are those created by adding a horizontal strip of size $p$ to $\lambda$. If $p\le \lambda_1+1-\lambda_2$ the images of the legal shapes are $\Delta$, otherwise their images are $\Delta\setminus\{(\gamma_2+p,\gamma_1)\}$. If $p\le \lambda_1+1-\lambda_2$ every element of $\Delta$ satisfies (1) and (2), otherwise every element of $\Delta$ satisfies (1) and (2) except for $(\gamma_2+p,\gamma_1)$ which fails (1). Then we are done by Lemma~\ref{lemma:nonthreshC}.  

($|\olambda|\le 2n-3$ and $p+|\olambda|>2n-3$:) By (B), the shapes in $\oalpha_p\star\olambda$ are those created by adding a horizontal strip of size $p$ to $\lambda$ and then removing a box from either the first or second row (to occupy the root of the top region). The images of the legal shapes are $\Delta\cup\{\gamma^*\}$. Every element of $\Delta$ satisfies (1) and (2), and also $\gamma\rightarrow\gamma^*$ by Lemma~\ref{lemma:threshC}. Then we are done by Lemma~\ref{lemma:removeboxC}.
\end{proof}

\subsection{Agreement of Definition~\ref{def:LGproduct} with Theorem \ref{TheoremBKTPieriC}} If $p=2n-2$, then $\oalpha_p=\langle 2n-3,0|\bullet\rangle$ and straightforwardly $\oalpha_p\star\olambda=0$ (and thus by Lemma~\ref{lemma:shapesagreeC} $\sigma_{p}\cdot\sigma_{\gamma}=0$) unless $\olambda=\langle\lambda|\circ\rangle$ and $\lambda_2=0$, i.e., $\olambda=\oalpha_q$ for some $q<2n-2$. Thus we may assume $p<2n-2$. Then by Lemma \ref{lemma:shapesagreeC} it suffices to show that for any (legal) $c\cdot\omu$ appearing in $\oalpha_p\star\olambda$ we have $c=2^{N(\gamma,\delta)}$, where $\delta=f(\omu)$. Since illegal terms do not contribute, and $f(\omu)\in P(n-2,n)$ if and only if $\omu$ is legal, we may assume the terms whose coefficients we examine below are legal. 

\noindent {\bf Case 1:} ($p+|\olambda|\le 2n-3$): By (A), the coefficient of each term in $\oalpha_p\star\olambda$ is $1$. Thus we must show the image $\delta$ of any term has $N(\gamma,\delta)=0$. Since $|\delta|\le 2n-3$, we have $\mathbb{D}=\mathbb{D}_1$. If $\gamma_1\ge n-1$, then since $2n-2-\gamma_1\le \gamma_1+1$ and $2n-3-\delta_2\ge \delta_1$, we have $\mathbb{D}_1\setminus S'_1=\emptyset$, so $N(\gamma,\delta)=0$. Suppose $\gamma_1<n-1$. If $\mathbb{D}_1=\emptyset$, then $N(\gamma,\delta)=0$. Otherwise, $(1:n-1)\in \mathbb{D}_1$ and is not killed, whence $N(\gamma,\delta)=0$ follows since by Lemma~\ref{lemma:bisectionC}, $\mathbb{D}_1$ is not bisected.

\noindent {\bf Case 2:} ($|\olambda|> 2n-3$): By (C), the coefficient of each term in $\oalpha_p\star\olambda$ is $1$. Thus we must show the image $\delta$ of any term has $N(\gamma,\delta)=0$. Since $2n-1-\gamma_1\le \gamma_1+1$, we have $\mathbb{D}_1\setminus T'_1 = \emptyset$, so only $\mathbb{D}_2$ can contribute to $\mathbb{A}$. If $\gamma_2\ge n-2$ then all boxes of $R$ in row $2$ except $(2:n-1)$ are mentioned in (1), hence $N(\gamma,\delta)=0$. Suppose $\gamma_2<n-2$. If $\mathbb{D}_2=\emptyset$, then $N(\gamma,\delta)=0$. Otherwise $(2:n-1)\in \mathbb{D}_2$ and is not killed, and then $N(\gamma,\delta)=0$ follows since by Lemma~\ref{lemma:bisectionC}, $\mathbb{D}_2$ is not bisected.

\noindent {\bf Case 3:} ($|\olambda|\le 2n-3$, $p+|\olambda|>2n-3$): Let $M=\min\{\lambda_1-\lambda_2,p\}$. Then by (B), we compute

\[\oalpha_p\star\olambda = \langle \lambda_1+p,\lambda_2-1|\bullet\rangle + 2\sum_{1\le j\le M}\langle \lambda_1+p-j,\lambda_2-1+j|\bullet\rangle + \langle \lambda_1+p-M-1,\lambda_2+M|\bullet\rangle.\]

First suppose $\delta=f(\langle \lambda_1+p,\lambda_2-1|\bullet\rangle)=\gamma^*$. Then $N(\gamma,\delta)=0$ by Lemmas~\ref{lemma:threshC} and \ref{lemma:specialPiericaseC}. 

Next, suppose $\delta$ is the image of a term in the summation. If $\gamma_1<n-1$, then since $\delta_2<\gamma_1$ a component of $\mathbb{D}$ is bisected by Lemma~\ref{lemma:bisectionC}. Thus $N(\gamma,\delta)=1$ by Corollary~\ref{cor:bisectionC}. Therefore, suppose $\gamma_1\ge n-1$. By Lemma~\ref{lemma:bisectionC} no component of $\mathbb{D}$ is bisected, and since $\delta_2<\gamma_1$ we have $\mathbb{D}_1$ is not connected to $\mathbb{D}_2$. Since $\gamma_2\le n-2$, if $\mathbb{D}_2\neq\emptyset$ then $(2:n-1)\in \mathbb{D}_2$ and is not killed, so $\mathbb{D}_2$ does not contribute to $N(\gamma,\delta)$. Since $\gamma_1\ge n-1$, we have $(1:n-1)\notin \mathbb{D}_1$, and since $\mathbb{D}_1\neq\emptyset$, by Lemma~\ref{lemma:threshcptsurvivesC} not every box of $\mathbb{D}_1$ is killed. Thus $\mathbb{D}_1$ contributes $1$ to $N(\gamma,\delta)$, whence $N(\gamma,\delta)=1$. 

Finally, suppose $\delta=f(\langle \lambda_1+p-M-1,\lambda_2+M|\bullet\rangle)$. Then either $\delta_2=\gamma_1$ or $\mathbb{D}_1=\emptyset$. If $\delta_2=\gamma_1$ then $\mathbb{D}=\mathbb{D}_1\cup\mathbb{D}_2$ is connected, and since $\gamma_2\le n-2$ it uses $(2:n-1)$. By Lemma~\ref{lemma:bisectionC} $\mathbb{D}$ is not bisected, hence $N(\gamma,\delta)=0$. Thus suppose $\mathbb{D}_1=\emptyset$. Then if also $\mathbb{D}_2=\emptyset$, we have $N(\gamma,\delta)=0$. Otherwise, since $\gamma_1\le n-2$ we have $(2:n-1)\in\mathbb{D}_2$, and $(2:n-1)$ is not killed. Then $N(\gamma,\delta)=0$ follows since by Lemma~\ref{lemma:bisectionC}, $\mathbb{D}_2$ is not bisected.

\section{Proof of Theorem~\ref{Thm:Pieriagreement}(II)}

We now follow \cite[pg. 31-33]{BKT:Inventiones}. The Schubert varieties of $OG(2,2n)$ are indexed by the set $\tilde{P}(n-2,n)$ of all pairs $\tilde{\gamma}=(\gamma; {\tt type}(\gamma))$, where $\gamma$ is an element of the set $P(n-2,n)$ of all $(n-2)$-strict partitions inside a $2\times (2n-3)$ rectangle, and also ${\tt type}(\gamma)=0$ if no part of $\gamma$ has size $n-2$ and ${\tt type}(\gamma)\in \{1,2\}$ otherwise. 
The {\bf Pieri classes} of \cite{BKT:Inventiones} are those indexed by $\tilde{\gamma}$ with $\gamma=(p,0)$. If $p \neq n-2$ then the class is denoted by $\sigma_p$. Otherwise if ${\tt type}(\gamma)=1$ (respectively, ${\tt type}(\gamma)=2$) the class is denoted $\sigma_{n-2}$ (respectively, $\sigma'_{n-2}$).

Fix an integer $p\in [1,2n-3]$, and suppose $\gamma, \delta\in P(n-2,n)$ with $|\delta|=|\gamma|+p$. Then the relation $\gamma \rightarrow \delta$ is defined as in the previous section, except now the box in row $r$ and column $c$ of $\gamma$ is {\bf related} to the box in row $r'$ and column $c'$ if $|c-(2n-3)/2|+r = |c'-(2n-3)/2|+r'$.

Define $\mathbb{A}$ as in the previous section. Then define $N'(\gamma,\delta)$ to be the number of connected components of $\mathbb{A}$ (respectively, one less than this number) if $p\le n-2$ (respectively, if $p>n-2$).

Let $g(\gamma,\delta)$ be how many of the first $n-2$ columns of $\delta$ have no $(\delta \setminus \gamma)$-boxes, and let $h(\tilde{\gamma}, \tilde{\delta})=g(\gamma,\delta) + \text{max}(\text{type}(\gamma),\text{type}(\delta))$. If $p \neq n-2$, set $\epsilon_{\tilde{\gamma} \tilde{\delta}}=1$. If $p=n-2$ and $N'(\gamma,\delta)>0$, set $\epsilon_{\tilde{\gamma} \tilde{\delta}}=\epsilon'_{\tilde{\gamma} \tilde{\delta}}=\frac{1}{2}$, while if $N'(\gamma,\delta)=0$, define

$\epsilon_{\tilde{\gamma} \tilde{\delta}} = \begin{cases}
1 & \text{if $h(\tilde{\gamma}, \tilde{\delta})$ is odd} \\
0 & \text{otherwise}
\end{cases}$ \qquad and \qquad
$\epsilon'_{\tilde{\gamma} \tilde{\delta}} = \begin{cases}
1 & \text{if $h(\tilde{\gamma}, \tilde{\delta})$ is even} \\
0 & \text{otherwise.}
\end{cases}$

Then the specialization of the Pieri rule of \cite[Theorem 3.1]{BKT:Inventiones} to the adjoint $OG(2,2n)$ is

\begin{Theorem}[\cite{BKT:Inventiones}]\label{TheoremBKTPieri}(Pieri rule for OG(2,2n))
For any $\tilde{\gamma} \in \tilde{P}(n-2,n)$ and integer $p \in [1,2n-3]$,
\begin{equation}\nonumber
\sigma_p \cdot \sigma_{\tilde{\gamma}} = \sum_{\tilde{\delta}} \epsilon_{\tilde{\gamma} \tilde{\delta}} 2^{N'(\gamma,\delta)}\sigma_{\tilde{\delta}}
\end{equation}
where the sum is over all $\tilde{\delta} \in \tilde{P}(n-2,n)$ with $\gamma \rightarrow \delta$ and ${\tt type}(\gamma)+{\tt type}(\delta) \neq 3$. 
Furthermore, the product $\sigma'_{n-2} \cdot \sigma_{\tilde{\gamma}}$ is obtained by replacing $\epsilon_{\tilde{\gamma} \tilde{\delta}}$ with $\epsilon'_{\tilde{\gamma} \tilde{\delta}}$ throughout.
\end{Theorem}

Let $(r:c)$ denote the box in row $r$, column $c$ of $2\times (2n-3)$. Let $L$ denote the first $n-2$ columns of $2\times(2n-3)$ and $R$ the latter $n-1$ columns. Given $\gamma,\delta\in P(n-2,n)$ with $|\delta|=|\gamma|+p$, recall from the previous section the definitions of $\mathbb{D}_1$, $\mathbb{D}_2$ and $\mathbb{D}$. 
Let $\gamma^*$ denote the shape $(\gamma_1+p+1,\gamma_2-1)$. The following three lemmas are proved similarly to (respectively) Lemmas~\ref{lemma:removeboxC}, \ref{lemma:threshC} and \ref{lemma:nonthreshC}.

\begin{Lemma}\label{lemma:removebox}
If $\gamma \rightarrow \delta$ and $\gamma \not\subseteq \delta$, then $\delta = \gamma^*$.
\end{Lemma}

\begin{Lemma}\label{lemma:thresh}
Suppose $|\gamma|\le 2n-4$ and $p+|\gamma|>2n-4$. If $\gamma^*\in P(n-2,n)$, then $\gamma \rightarrow \gamma^*$.
\end{Lemma}

\begin{Lemma}\label{lemma:nonthresh}
If either $|\delta|\le 2n-4$ or $|\gamma|>2n-4$, then $\gamma \rightarrow \delta \Rightarrow \gamma \subseteq \delta$. In particular, $\delta$ is obtained from $\gamma$ without removing any box of $\gamma$.
\end{Lemma}

Given $\gamma\rightarrow\delta$, recall from the previous section the definition of when a box of $\mathbb{D}$ is {\bf killed} and when a connected component $\mathbb{D}$ is {\bf bisected}. If also $\gamma\subset\delta$, recall the definitions of $S$ and $T$.

\begin{Lemma}\label{lemma:specialPiericase}
If $\gamma^*\in P(n-2,n)$ and $\gamma\rightarrow\gamma^*$, then $N'(\gamma,\delta)=1$ if $\gamma_1<n-2$ and $p\le n-2$, and $N'(\gamma,\delta)=0$ otherwise.
\end{Lemma}
\begin{proof}
Let $\delta=\gamma^*$. If $\gamma_1\ge n-2$, all boxes of $R$ are mentioned in (1) or (2), so $N'(\gamma,\delta)=0$. Suppose $\gamma_1<n-2$. Then $\mathbb{D}_2=\emptyset$, so $\mathbb{D}=\mathbb{D}_1$. Here $(1:n-1)$ is a $\mathbb{D}_1$-box and is not killed. By (1), (2) it is clear the $\mathbb{D}_1$-boxes killed are the last $l$ boxes of $\mathbb{D}_1$ for some $l>0$, hence $\mathbb{D}_1$ is not bisected. Thus $N'(\gamma,\delta)=0$ if $p>n-2$, and $N'(\gamma,\delta)=1$ if $p\le n-2$.
\end{proof}

\begin{Lemma}\label{lemma:killed} 
Let $\gamma\rightarrow\delta$ with $\gamma\subset\delta$.
Suppose $(1:c)\in\mathbb{D}_1$. Then 
\begin{itemize}
\item $(1:c)$ is killed by $S$ if and only if $(1:c)\in S'_1 = \{(1:c') : 2n-3-\gamma_1\le c'\le 2n-4-\delta_2\}$
\item $(1:c)$ is killed by $T$ if and only if $(1:c) \in T'_1 = \{(1:c') : 2n-2-\gamma_2\le c'\le 2n-3\}$
\end{itemize}
Suppose $(2:c)\in\mathbb{D}_2$. Then
\begin{itemize}
\item $(2:c)$ is never killed by $S$
\item $(2:c)$ is killed by $T$ if and only if $(2:c) \in T'_2 = \{(2:c') : 2n-3-\gamma_2\le c'\le 2n-4\}$
\end{itemize}
\end{Lemma}
\begin{proof}
That $(2:c)$ is never killed by $S$ follows since the existence of a $\mathbb{D}$-box in row 2 implies $\delta_2>n-2$ and thus $S=\emptyset$. The remaining points follow from the definition of being related.
\end{proof}

\begin{Corollary}\label{cor:killed}
Suppose $\gamma\rightarrow\delta$ with $\gamma\subset\delta$. Then if $(1:2n-3-\delta_2)$ is a $\mathbb{D}_1$-box, it is not killed.
\end{Corollary}
\begin{proof}
Since $2n-4-\delta_2<2n-3-\delta_2<2n-2-\gamma_2$, $(1:2n-3-\delta_2)$ is not in $S'_1$ or $T'_1$. 
\end{proof}

\begin{Lemma}\label{lemma:bisection}
A connected component of $\mathbb{D}$ is bisected if and only if all of the following hold:
\begin{itemize}
\item[(i)] $|\gamma|\le 2n-4$ and $|\delta|>2n-4$
\item[(ii)] $\gamma\subseteq\delta$
\item[(iii)] $\gamma_1<n-2$
\item[(iv)] $\delta_2<\gamma_1.$
\end{itemize} 
\end{Lemma}
\begin{proof}
Similar to the proof of Lemma~\ref{lemma:bisectionC}, using Corollary~\ref{cor:killed}.
\end{proof}

\begin{Corollary}\label{cor:bisection}
If a connected component of $\mathbb{D}$ is bisected, then $\mathbb{A}$ has two connected components.
\end{Corollary}
\begin{proof}
Similarly to the proof of Lemma~\ref{lemma:bisectionC}, if a connected component of $\mathbb{D}$ is bisected then $\mathbb{D}=\mathbb{D}_1$, so $\mathbb{D}_1$ is bisected. It also follows from the proof that $\mathbb{D}_1\setminus(S'_1\cup T'_1)=\mathbb{A}$ has two connected components.
\end{proof}

\begin{Lemma}\label{lemma:threshcptsurvives}
If $\gamma\rightarrow\delta$ with $\gamma\subset\delta$, $|\gamma|\le 2n-4$, $|\delta|>2n-4$, $\gamma_1\ge n-2$ and also $\mathbb{D}_1$ is nonempty, then not all $\mathbb{D}_1$-boxes are killed.
\end{Lemma}
\begin{proof}
Similar to the proof of Lemma~\ref{lemma:threshcptsurvivesC}.
\end{proof}

As in the previous section, we will use the notation of \cite{Searles.Yong} for RYDs for $OG(2,2n)$ from now on. An RYD will be denoted $\olambda = \langle \lambda|\bullet\rangle$ or $\olambda = \langle \lambda|\circ\rangle$ where $\lambda=(\lambda_1,\lambda_2)$ is the partition in $2\times (2n-4)$ corresponding to the roots used in the base region, and $\bullet/\circ$ denotes whether $\olambda$ uses the single root in the top region or not. If neither $\lambda_1$ nor $\lambda_2$ is equal to $n-2$, then $\olambda$ is said to be {\bf neutral}, otherwise $\olambda$ is {\bf charged} and is assigned a ``charge'' denoted ${\rm ch}(\olambda)$, which is either $\uparrow$ or $\downarrow$ exactly as in the introduction. 
Let $\Pi(\olambda)$ denote $\langle \lambda_1,\lambda_2|{\bullet}/{\circ}\rangle$, i.e., ignoring any charge. For shapes $\olambda, \omu \in {\mathbb Y}_{OG(2,2n)}$, let $M=\min\{\lambda_1-\lambda_2, \mu_1-\mu_2\}$. We reprise the definition of the product $\star$ on RYDs from \cite{Searles.Yong}:

\begin{Definition}\cite[Definition 5.1]{Searles.Yong}
\label{def:fakeproduct}
For $\olambda, \omu \in {\mathbb Y}_{OG(2,2n)}$, define an expression $\Pi(\olambda)\diamond\Pi(\omu)$:
\begin{enumerate}
\item[(A)] If $|\langle\lambda| \circ\rangle|+|\langle\mu| \circ\rangle| \le 2n-4$, then
\[\Pi(\langle\lambda|\circ\rangle) \diamond \Pi(\langle\mu|\circ\rangle) = \sum_{0 \le k \le M}
\langle\lambda_1+\mu_1-k, \lambda_2+\mu_2+k| \circ\rangle\]
\item[(B)] If $|\langle\lambda|\circ\rangle| + |\langle\mu|\circ\rangle| > 2n-4$, then
\begin{multline}\nonumber
\Pi(\langle\lambda|\circ\rangle)\diamond \Pi(\langle\mu|\circ\rangle) = \langle\lambda_1+\mu_1, \lambda_2+\mu_2-1| \bullet\rangle +
2\sum_{1 \le k \le M}  \langle\lambda_1+\mu_1-k, \lambda_2+\mu_2+k-1| \bullet\rangle\\ + \langle\lambda_1+\mu_1-M-1, \lambda_2+\mu_2+M|\bullet\rangle
\end{multline}
\item[(C)] $\Pi(\langle\lambda|\bullet\rangle)\diamond \Pi(\langle\mu|\circ\rangle) = \Pi(\langle\lambda|\circ\rangle)\diamond \Pi(\langle\mu|\bullet\rangle) =
\sum_{0 \le k \le M}  \langle\lambda_1+\mu_1-k, \lambda_2+\mu_2+k| \bullet\rangle$
\item[(D)] $\Pi(\langle\lambda|\bullet\rangle) \diamond \Pi(\langle\mu|\bullet\rangle) = 0$.
\end{enumerate}

Declare any $\overline{\alpha}$ in the above expressions to be zero if
$(\alpha_1,\alpha_2)$ is not a partition in $2\times(2n-4)$. Such $\overline{\alpha}$ will be called {\bf illegal}.
\end{Definition}

If $\olambda, \omu$ are both charged, we say they {\bf match} if ${\rm ch}(\olambda)={\rm ch}(\omu)$, and are {\bf opposite} otherwise. The opposite charge to ${\rm ch}(\olambda)$ is denoted ${\rm op}(\olambda)$. Define:

$\eta_{\olambda,\omu}=
\begin{cases}
2 & \text{if $\olambda$, $\omu$ are charged and match and $n$ is even;}\\
2 & \text{if $\olambda$, $\omu$ are charged and opposite and $n$ is odd;}\\
1 & \text{if $\olambda$ or $\omu$ are not charged;}\\
0 & \text{otherwise}
\end{cases}$

If a $\okappa$ appearing in $\Pi(\olambda)\diamond\Pi(\omu)$ has $\kappa_1=n-2$ or $\kappa_2=n-2$, we say $\okappa$ is {\bf ambiguous}. We say $\olambda\in\mathbb{Y}_{OG(2,2n)}$ is {\bf Pieri} if $\Pi(\olambda)=\langle j,0|{\bullet}/{\circ}\rangle$, and {\bf non-Pieri} otherwise.

\begin{Definition}\label{def:starproduct}\cite[Definition 5.2]{Searles.Yong}
Let $\olambda,\omu\in {\mathbb Y}_{OG(2,2n)}$. Define a commutative product $\star$ on $R={\mathbb Z}[{\mathbb Y}_{OG(2,2n)}]$:

If $\Pi(\olambda)=\Pi(\omu)=\langle n-2,0|\circ\rangle$, then

$\olambda\star\omu=\begin{cases}
\sum_{0\le k\le \frac{n-2}{2}}\langle 2n-4-2k, 2k|\circ\rangle & \text{if $n$ is even and $\olambda$, $\omu$ match}\\                                       
\sum_{0\le k\le \frac{n-4}{2}}\langle 2n-5-2k, 2k+1|\circ\rangle & \text{if $n$ is even and $\olambda$, $\omu$ are opposite}\\
\sum_{0\le k\le \frac{n-3}{2}}\langle 2n-5-2k, 2k+1|\circ\rangle & \text{if $n$ is odd and $\olambda$, $\omu$ match} \\                                    
\sum_{0\le k\le \frac{n-3}{2}}\langle 2n-4-2k, 2k|\circ\rangle & \text{if $n$ is odd and $\olambda$, $\omu$ are opposite} 
\end{cases}$

\noindent
where for the first and third cases above, the shape $\langle n-2,n-2|\circ\rangle$ is assigned ${\rm ch}(\olambda)={\rm ch}(\omu)$.

Otherwise, compute $\Pi(\olambda)\diamond\Pi(\omu)$ and
\begin{itemize}
\item[(i)] First, replace any term $\okappa$ that has $\kappa_1=2n-4$ by
$\eta_{\olambda,\omu}\okappa$.
\item[(ii)] Next, replace each $\okappa$ by $2^{{\rm fsh}(\okappa)-{\rm fsh}(\olambda)-{\rm fsh}(\omu)}\okappa$.
\item[(iii)] Lastly, ``disambiguate'' using one in the following complete list of possibilities:
\begin{itemize}
\item[(iii.1)] \noindent{\sf (if $\olambda, \omu$ are both non-Pieri)} 
Replace any ambiguous $\okappa$ by $\frac{1}{2}(\okappa^{\uparrow}+\okappa^{\downarrow})$.

\item[(iii.2)] \noindent{\sf (if one of $\olambda, \omu$ is  neutral and Pieri)} 
Since $\Pi(\olambda)\diamond\Pi(\omu) = \Pi(\omu)\diamond\Pi(\olambda)$, we may assume $\olambda$ is Pieri. Then replace any ambiguous $\okappa$ by $\frac{1}{2}(\okappa^{\uparrow}+\okappa^{\downarrow})$ if $\omu$ is neutral, and by $\okappa^{{\rm ch}(\omu)}$ if $\omu$ is charged.

\item[(iii.3)] \noindent{\sf (if one of $\olambda, \omu$ is  charged and Pieri, and the other is non-Pieri)}. As above, we may assume $\olambda$ is Pieri. In particular, $\Pi(\olambda)=\langle n-2,0|\circ\rangle$. 
\begin{itemize}
\item[(iii.3a)] If $\omu = \langle \mu |\bullet\rangle$ is neutral and $|\mu|=2n-4$,
then replace the ambiguous term $\langle 2n-4, n-2|\bullet\rangle$ by $\langle 2n-4, n-2|\bullet\rangle^{{\rm ch}(\olambda)}$ if $\mu_1$ is even and by $\langle 2n-4, n-2|\bullet\rangle^{{\rm op}(\olambda)}$ if $\mu_1$ is odd.
\item[(iii.3b)] Otherwise, replace any ambiguous $\okappa$ by $\frac{1}{2}(\okappa^{\uparrow}+\okappa^{\downarrow})$ if $\omu$ is 
neutral, and by $\okappa^{{\rm ch}(\omu)}$ if $\omu$ is charged.
\end{itemize}
\end{itemize}
\end{itemize}
\end{Definition}

Define

$f(\Pi(\olambda))=\begin{cases}
(\lambda_1,\lambda_2)\in P(n-2,n) & \text{if $\olambda=\langle \lambda|\circ\rangle$}\\
(\lambda_1+1,\lambda_2)\in P(n-2,n) & \text{if $\olambda=\langle \lambda|\bullet\rangle$.}
\end{cases}$

Then the following specializes Proposition~\ref{prop:RYDtoBKTD} to the adjoint case $k=2$, where we write $F$ instead of $F_2$:

\begin{Proposition} The elements of ${\mathbb Y}_{OG(2,2n)}$ are in bijection with the elements of $\tilde{P}(n-2,n)$ via

$F(\olambda)=\begin{cases}
(f(\Pi(\olambda));0) & \text{if $\olambda$ is neutral} \\
(f(\Pi(\olambda));1) & \text{if $\olambda$ is assigned $\uparrow$} \\
(f(\Pi(\olambda));2) & \text{if $\olambda$ is assigned $\downarrow$} \\
\end{cases}$
\end{Proposition}

Let $\oalpha_p$ denote $\langle p,0|{\bullet}/{\circ}\rangle\in \mathbb{Y}_{OG(2,2n)}$. Throughout, given $\olambda\in\mathbb{Y}_{OG(2,2n)}$ let $\gamma$ denote $f(\Pi(\olambda))$ and $\tilde{\gamma}$ denote $F(\olambda)$. 

\begin{Lemma}\label{lemma:shapesagree}
Suppose $p\neq 2n-3$. Then a (legal) shape $\okappa$ appears in the expansion $\Pi(\oalpha_p)\diamond\Pi(\olambda)$ if and only if $\gamma\rightarrow f(\okappa)$. If also $p\neq n-2$, then a (legal) shape $\omu$ appears in the expansion $\oalpha_p\star\olambda$ if and only if $F(\omu)$ appears in the expansion $\sigma_p\cdot\sigma_{\tilde{\gamma}}$.
\end{Lemma}
\begin{proof}
The first claim is proved similarly to the proof of Lemma~\ref{lemma:shapesagreeC}. Now suppose $p\neq n-2$. Then (i) has no effect on $\Pi(\oalpha_p)\diamond\Pi(\olambda)$, and (ii) multiplies every term by a nonzero coefficient. Then terms are disambiguated by (iii.2). Under $F$, (iii.2) translates exactly to the condition ${\tt type}(\gamma)+{\tt type}(\delta)\neq 3$. So the charge assignments in $\oalpha_p\star\olambda$ agree with the types appearing in $\sigma_p \cdot \sigma_{\tilde{\gamma}}$. This proves the second claim.
\end{proof}

The following lemma from \cite{Searles.Yong} will be used in the proof.

\begin{Lemma}
\label{lemma:firstrow}
If $\okappa=\langle\kappa_1,\kappa_2|{\bullet}{/\circ}\rangle$ appears in $\Pi(\olambda)\diamond\Pi(\omu)$ then
\[\kappa_1\geq
\begin{cases}
\max(\lambda_1+\mu_2,\lambda_2+\mu_1) & \mbox{if $\Pi(\olambda)\diamond\Pi(\omu)$ is described by (A) or (C)}\\
\max(\lambda_1+\mu_2,\lambda_2+\mu_1)-1 & \mbox{if $\Pi(\olambda)\diamond\Pi(\omu)$ is described by (B)}
\end{cases}\]
\end{Lemma}

\subsection{Agreement of Definition~\ref{def:starproduct} with Theorem \ref{TheoremBKTPieri} when $p>n-2$.} 
Suppose $p=2n-3$. Then $\oalpha_p=\langle 2n-4,0|\bullet\rangle$ and by Lemma~\ref{lemma:firstrow} or by (D) $\oalpha_p\star\olambda=0$ unless $\olambda=\langle\lambda|\circ\rangle$ and $\lambda_2=0$, in which case $\oalpha_p\star\olambda=\langle 2n-4,\lambda_1|\bullet\rangle$ (assigned ${\rm ch}(\olambda)$ if $\lambda_1=n-2$). Clearly the only $\delta$ with $\gamma\rightarrow\delta$ is $\delta=(2n-3,\lambda_1)=f(\langle 2n-4,\lambda_1|\bullet\rangle)$. We have $N'(\gamma,\delta)=0$ since $\mathbb{D}=\mathbb{D}_1\cup\mathbb{D}_2$ is connected. Finally, if $\lambda_1=n-2$ then only $(\delta;{\tt type}(\gamma))$ appears in $\sigma_{2n-3}\cdot\sigma_{\tilde{\gamma}}$, since ${\tt type}(\gamma)+{\tt type}(\delta)\neq 3$.  

Thus assume $p<2n-3$. By Lemma \ref{lemma:shapesagree} and since $\epsilon_{\gamma,\delta}=1$, it suffices to show that for any (legal) $c\cdot\omu$ appearing in $\oalpha_p\star\olambda$, $c=2^{N'(\gamma,\delta)}$, where $\tilde{\delta}=F(\omu)$. As in the previous section, we may assume terms whose coefficients we examine below are legal.

\noindent {\bf Case 1:} ($p+|\olambda|\le 2n-4$): Then $\oalpha_p\star\olambda = \sum_{0\le j \le \lambda_1-\lambda_2}\langle \lambda_1+p-j,\lambda_2+j|\circ\rangle$ (neutral). 
For the image $\tilde{\delta}$ of any term, since $|\delta|\le 2n-4$ we have $\mathbb{D}_2=\emptyset$ and so $\mathbb{D}=\mathbb{D}_1$. By Lemma~\ref{lemma:bisection} $\mathbb{D}_1$ is not bisected, so $N'(\gamma,\delta)=0$.

\noindent {\bf Case 2:} ($|\olambda|>2n-4$): We may assume $\lambda_2<n-2$, since otherwise $\Pi(\oalpha_p)\diamond\Pi(\olambda)=0$ by Lemma~\ref{lemma:firstrow}. Then $\oalpha_p\star\olambda = \sum_{0\le j \le \lambda_1-\lambda_2}\langle \lambda_1+p-j,\lambda_2+j|\bullet\rangle$ (neutral). For the image $\tilde{\delta}$ of any term, since $\gamma_1>n-2$ and $2n-2-\gamma_2 \le \gamma_1+1$ we have $\mathbb{D}_1\setminus T'_1=\emptyset$. By Lemma~\ref{lemma:bisection} there is no bisection, thus $N'(\gamma,\delta)= 0$. 

\noindent {\bf Case 3:} ($|\olambda|\le 2n-4$, $p+|\olambda|>2n-4$): We need three subcases.

\noindent {\bf Subcase 3a:} ($\lambda_1<n-2$): We compute 
\[\oalpha_p\star\olambda = \langle \lambda_1+p,\lambda_2-1|\bullet\rangle + 2\sum_{1\le j \le \lambda_1-\lambda_2} \langle\lambda_1+p-j,\lambda_2-1+j|\bullet\rangle + \langle\lambda_2+p-1,\lambda_1|\bullet\rangle \text{\ (neutral)}.\]  
If $\tilde{\delta}=F(\langle \lambda_1+p,\lambda_2-1|\bullet\rangle)=\tilde{\gamma}^*$ then $N'(\gamma,\delta)= 0$ by Lemmas~\ref{lemma:thresh} and \ref{lemma:specialPiericase}. For the image $\tilde{\delta}$ of a term in the summation, since $\delta_2<\gamma_1$ a component of $\mathbb{D}$ is bisected by Lemma~\ref{lemma:bisection}. Thus $N'(\gamma,\delta)=1$ by Corollary~\ref{cor:bisection}.
If $\tilde{\delta}=F(\langle\lambda_2+p-1,\lambda_1|\bullet\rangle)$, then $\delta_2=\gamma_1<n-2$ so $\mathbb{D}=\mathbb{D}_1$. Then $N'(\gamma,\delta)=0$ by Lemma~\ref{lemma:bisection}. 

\noindent {\bf Subcase 3b:} ($\lambda_1>n-2$): Let $M=\min\{\lambda_1-\lambda_2,p\}$. We compute \[\Pi(\oalpha_p)\diamond\Pi(\olambda) = \langle \lambda_1+p,\lambda_2-1|\bullet\rangle + 2\sum_{1\le j \le M} \langle\lambda_1+p-j,\lambda_2-1+j|\bullet\rangle + \langle\lambda_1+p-M-1,\lambda_2+M|\bullet\rangle.\] 
The first term is illegal. Next, (ii) multiplies any term $\okappa$ by $\frac{1}{2}$ if $\kappa_2<n-2$, and by $1$ otherwise. If a $\okappa$ is ambiguous, by (iii.2) it splits. 

Thus for the image $\delta$ of a term in the summation, we must $N'(\gamma,\delta)= 0$ if $\delta_2 \le n-2$ and $N'(\gamma,\delta)= 1$ if $\delta_2>n-2$. Assume $\delta_2\le n-2$. Then $\mathbb{D}=\mathbb{D}_1$, and $N'(\gamma,\delta)= 0$ follows from Lemma~\ref{lemma:bisection}. 
Now assume $\delta_2>n-2$. Then $\mathbb{D}=\mathbb{D}_1\cup\mathbb{D}_2$, where $\mathbb{D}_1,\mathbb{D}_2\neq\emptyset$ and $\mathbb{D}_1$ is not connected to $\mathbb{D}_2$. Then $N'(\gamma,\delta)= 1$ by Lemma~\ref{lemma:bisection}, Lemma~\ref{lemma:threshcptsurvives} and the fact that (since $\gamma_2<n-2$), $(2:n-1)\in \mathbb{D}_2\setminus T'_2$. If $\delta=f(\langle\lambda_1+p-M-1,\lambda_2+M|\bullet\rangle)$, we have $\mathbb{D}=\mathbb{D}_1\cup\mathbb{D}_2$ is connected. Then $N'(\gamma,\delta)= 0$ by Lemma~\ref{lemma:bisection}.

\noindent {\bf Subcase 3c:} ($\lambda_1=n-2$): We compute 
\[\oalpha_p\star\olambda = \sum_{1\le j \le n-2-\lambda_2} \langle n-2+p-j,\lambda_2-1+j|\bullet\rangle + \langle\lambda_2+p-1,n-2|\bullet\rangle^{{\rm ch}(\olambda)}.\] 
For the image $\tilde{\delta}$ of each term, since $\delta_2\le n-2$ we have $\mathbb{D}=\mathbb{D}_1$. Then $N'(\gamma,\delta)= 0$ by Lemma~\ref{lemma:bisection}.

\subsection{Agreement of Definition~\ref{def:starproduct} with Theorem \ref{TheoremBKTPieri} when $p<n-2$.} 

By Lemma \ref{lemma:shapesagree} and since $\epsilon_{\gamma,\delta}=1$, it suffices to show that for any $c\cdot\omu$ appearing in $\oalpha_p\star\olambda$, $c=2^{N'(\gamma,\delta)}$, where $\tilde{\delta}=F(\omu)$.

\noindent {\bf Case 1:} ($p+|\olambda|\le 2n-4$): There are two subcases.

\noindent {\bf Subcase 1a:} ($\lambda_1\ge n-2$): We compute $\oalpha_p\star\olambda = \sum_{0\le j\le p}\langle \lambda_1+p-j,\lambda_2+j|\circ\rangle$, where any term with first entry $n-2$ is assigned ${\rm ch}(\olambda)$. For the image $\tilde{\delta}$ of any term, since $|\delta|\le 2n-4$ we have $\delta_2\le n-2$ and $\mathbb{D}=\mathbb{D}_1$. Since $2n-3-\gamma_1\le \gamma_1+1$ and $2n-4-\delta_2\ge \delta_1$, we have $\mathbb{D}_1\setminus S'_1=\emptyset$, so $N'(\gamma,\delta)=0$.

\noindent {\bf Subcase 1b:} ($\lambda_1 < n-2$): Let $M=\min\{\lambda_1-\lambda_2,p\}$. We compute $\Pi(\oalpha_p)\diamond\Pi(\olambda) = \sum_{0\le j\le M}\langle \lambda_1+p-j,\lambda_2+j|\circ\rangle$. Now, (i) has no effect, and (ii) multiplies a term $\okappa$ by $1$ if $\kappa_1< n-2$, and by $2$ otherwise. If a $\okappa$ is ambiguous, it splits by (iii.2). Thus if $\delta=f(\okappa)$, we must show $N'(\gamma,\delta)=0$ if $\delta_1\le n-2$ and $N'(\gamma,\delta)=1$ if $\delta_1 > n-2$. 
If $\delta_1\le n-2$ then $\mathbb{D}=\emptyset$ so $N'(\gamma,\delta)=0$. Suppose $\delta_1 > n-2$. 
Then since $\delta_2<n-2$, we have $\mathbb{D}=\mathbb{D}_1$. Since $\delta_1>n-2$ and $\gamma_1<n-2$, we have $(1:n-1)\in\mathbb{D}_1$ and is not killed. Then $N'(\gamma,\delta)=1$ follows from Lemma~\ref{lemma:bisection}.

\noindent {\bf Case 2:} ($|\olambda|> 2n-4$): Let $M=\min\{\lambda_1-\lambda_2,p\}$. There are two subcases. 

\noindent {\bf Subcase 2a:} ($\lambda_2\ge n-2$): Here $\oalpha_p\star\olambda=\sum_{0\le j\le M}\langle \lambda_1+p-j,\lambda_2+j|\bullet\rangle$, where any charged term has charge ${\rm ch}(\olambda)$. For the image $\tilde{\delta}$ of any term, since $\gamma_2 \ge n-2$ all boxes of $R$ except $(1:n-1)$ are mentioned in (1). Since $(1:n-1)$ is not a $\mathbb{D}$-box, we have $N'(\gamma,\delta)=0$.

\noindent {\bf Subcase 2b:} ($\lambda_2 < n-2$): Here, $\Pi(\oalpha_p)\diamond\Pi(\olambda) = \sum_{0\le j\le M}\langle \lambda_1+p-j,\lambda_2+j|\bullet\rangle$. Then (ii) multiplies a term $\okappa$ by $1$ if $\kappa_2< n-2$, and by $2$ otherwise. If a $\okappa$ is ambiguous, by (iii.2) it splits. Therefore, if $\delta=f(\okappa)$, we must show that $N'(\gamma,\delta)=0$ if $\delta_2\le n-2$ and $N'(\gamma,\delta)=1$ if $\delta_2 > n-2$. For any $\delta$, since $2n-2-\gamma_2\le \gamma_1+1$ we have $\mathbb{D}_1\setminus T'_1=\emptyset$. Thus if $\delta_2\le n-2$, then $\mathbb{D}_2=\emptyset$, so $\mathbb{A}=\emptyset$ and $N'(\gamma,\delta)=0$. If $\delta_2>n-2$ then $(2:n-1)\in \mathbb{D}_2$ is not killed. Then $N'(\gamma,\delta)=1$ by Lemma~\ref{lemma:bisection}.

\noindent {\bf Case 3:} ($|\olambda|\le 2n-4$, $p+|\olambda|>2n-4$): Let $M=\min\{\lambda_1-\lambda_2,p\}$. There are three subcases.

\noindent {\bf Subcase 3a:} ($\lambda_1<n-2$): We compute 
\[\oalpha_p\star\olambda = 2\langle \lambda_1+p,\lambda_2-1|\bullet\rangle + 4\sum_{1\le j\le \lambda_1-\lambda_2}\langle \lambda_1+p-j,\lambda_2-1+j|\bullet\rangle + 2\langle \lambda_2+p-1,\lambda_1|\bullet\rangle \text{\ (neutral).}\] 
If $\tilde{\delta}=F(\langle \lambda_1+p,\lambda_2-1|\bullet\rangle)=\tilde{\gamma}^*$ then $N'(\gamma,\delta)=1$ by Lemmas~\ref{lemma:thresh} and \ref{lemma:specialPiericase}. For the image $\tilde{\delta}$ of a term in the summation, since $\delta_2<\gamma_1$ a component of $\mathbb{D}$ is bisected by Lemma~\ref{lemma:bisection}. Thus $N'(\gamma,\delta)=2$ by Corollary~\ref{cor:bisection}.
If $\tilde{\delta}=F(\langle \lambda_2+p-1,\lambda_1|\bullet\rangle)$ then $\delta_2=\gamma_1<n-2$ and $\delta_1>n-2$, so $\mathbb{D}=\mathbb{D}_1$ and $(1:n-1)\in\mathbb{D}_1$ is not killed. Then $N'(\gamma,\delta)=1$ by Lemma~\ref{lemma:bisection}. 

\noindent {\bf Subcase 3b:} ($\lambda_1 > n-2$): We compute 
\[\Pi(\oalpha_p)\diamond\Pi(\olambda) = \langle \lambda_1+p,\lambda_2-1|\bullet\rangle + 2\sum_{1\le j\le M}\langle \lambda_1+p-j,\lambda_2-1+j|\bullet\rangle + \langle \lambda_1+p-M-1,\lambda_2+M|\bullet\rangle.\] 
Then (ii) multiplies each term $\okappa$ of $\Pi(\oalpha_p)\diamond\Pi(\olambda)$ by $1$ if $\kappa_2<n-2$ and by $2$ otherwise, after which (iii.2) splits any ambiguous $\okappa$. If $\delta=f(\langle \lambda_1+p,\lambda_2-1|\bullet\rangle)=\gamma^*$ then $N'(\gamma,\delta)=0$ by Lemmas~\ref{lemma:thresh} and \ref{lemma:specialPiericase}. If $\delta=f(\langle \lambda_1+p-M-1,\lambda_2+M|\bullet\rangle)$ then either $\delta_2=\gamma_1$ or $\mathbb{D}_1=\emptyset$. If $\delta_2=\gamma_1$, then $N'(\gamma,\delta)=1$ follows by Lemma~\ref{lemma:bisection} and the fact that $(2:n-1)\in \mathbb{D}_2$ is not killed. If $\mathbb{D}_1=\emptyset$, then if $\delta_2\le n-2$ we have $\mathbb{D}_2=\emptyset$ and so $N'(\gamma,\delta)=0$, while if $\delta_2>n-2$ then by Lemma~\ref{lemma:bisection} and the fact that $(2:n-1)\in \mathbb{D}_2$ is not killed, we have $N'(\gamma,\delta)=1$.

For the image $\delta$ of a term in the summation, we must show $N'(\gamma,\delta)=1$ if $\delta_2 \le n-2$ and $N'(\gamma,\delta)=2$ if $\delta_2 > n-2$. If $\delta_2\le n-2$ then $\mathbb{D}=\mathbb{D}_1\neq\emptyset$, whence $N'(\gamma,\delta)=1$ by Lemma~\ref{lemma:bisection} and Lemma~\ref{lemma:threshcptsurvives}. If $\delta_2>n-2$, then since $\delta_2<\gamma_1$ we have $\mathbb{D}=\mathbb{D}_1\cup\mathbb{D}_2$, where $\mathbb{D}_1, \mathbb{D}_2\neq\emptyset$ and $\mathbb{D}_1$ is not connected to $\mathbb{D}_2$. Then $N'(\gamma,\delta)=2$ follows by Lemma~\ref{lemma:bisection}, Lemma~\ref{lemma:threshcptsurvives} and the fact that (since $\gamma_2<n-2$), $(2:n-1)\in \mathbb{D}_2\setminus T'_2$.

\noindent {\bf Subcase 3c:} ($\lambda_1=n-2$): We compute 
\[\oalpha_p\star\olambda = \langle n-2+p,\lambda_2-1|\bullet\rangle + 2\sum_{1\le j\le n-2-\lambda_2}\langle n-2+p-j,\lambda_2-1+j|\bullet\rangle + 2\langle \lambda_2+p-1,n-2|\bullet\rangle^{{\rm ch}(\olambda)}.\] 
If $\tilde{\delta}=F(\langle n-2+p,\lambda_2-1|\bullet\rangle)$ then $N'(\gamma,\delta)=0$ by Lemmas~\ref{lemma:thresh} and \ref{lemma:specialPiericase}. The image $\tilde{\delta}$ of any other term has $\delta_2\le n-2$ and $\delta_1>n-2$, so $\mathbb{D}=\mathbb{D}_1\neq\emptyset$. Then $N'(\gamma,\delta)=1$ by Lemma~\ref{lemma:bisection} and Lemma~\ref{lemma:threshcptsurvives}.

\subsection{Agreement of Definition~\ref{def:starproduct} with Theorem \ref{TheoremBKTPieri} when $p=n-2$.} 

It suffices to prove this for $\sigma_{n-2}=F(\langle n-2,0|\circ\rangle^\uparrow)$, since the proof for $\sigma'_{n-2}=F(\langle n-2,0|\circ\rangle^\downarrow)$ is essentially identical. 

\noindent {\bf Case 1:} ($\Pi(\olambda)=\langle n-2,0|\circ\rangle$): We compute $\sigma_{n-2}\cdot \sigma_{\tilde{\gamma}}$. Straightforwardly, $\gamma \rightarrow \delta$ if and only if $\delta \in \{(2n-4-j,j) : 0\le j \le n-2 \}$. Then the $\tilde{\delta}$ that can appear in $\sigma_{n-2}\cdot \sigma_{\tilde{\gamma}}$ are $(\delta;0)$ for all $\delta$ with with $\delta_2<n-2$, and $((n-2,n-2);{\tt type}(\gamma))$ (since ${\tt type}(\gamma)+{\tt type}(\delta)\neq 3$). For all such $\tilde{\delta}$ every $\mathbb{D}$-box is killed, so $N'(\gamma,\delta)=0$. We have $g(\gamma,\delta)=n-2-\delta_2$, so $h(\gamma,\delta)=n-2-\delta_2+{\tt type}(\gamma)$. Thus if $n$ is even and ${\rm type}(\gamma)=1$ or if $n$ is odd and ${\rm type}(\gamma)=2$, we have $\epsilon_{\tilde{\gamma},\tilde{\delta}}=1$ for all $\tilde{\delta}$ with $\delta_2$ even and $\epsilon_{\tilde{\gamma},\tilde{\delta}}=0$ for all $\tilde{\delta}$ with $\delta_2$ odd. Likewise, if $n$ is even and ${\rm type}(\gamma)=2$ or if $n$ is odd and ${\rm type}(\gamma)=1$, we have $\epsilon_{\tilde{\gamma},\tilde{\delta}}=1$ for all $\tilde{\delta}$ with $\delta_2$ odd and $\epsilon_{\tilde{\gamma},\tilde{\delta}}=0$ for all $\tilde{\delta}$ with $\delta_2$ even. This agrees with the definition (Definition~\ref{def:starproduct}) of $\langle n-2,0|\circ\rangle^\uparrow\star\langle n-2,0|\circ\rangle^{{\rm ch}(\olambda)}$.

In the remaining cases, we use Lemma~\ref{lemma:shapesagree}. We may assume $\lambda_2\neq 0$, since otherwise agreement follows by the previous case or previous subsections.

\noindent {\bf Case 2:} ($n-2+|\olambda|\le 2n-4$ and $\Pi(\olambda)\neq\langle n-2,0|\circ\rangle$): We compute $\langle n-2,0|\circ\rangle^\uparrow\star\olambda = \sum_{0\le j \le \lambda_1-\lambda_2}\langle n-2+\lambda_1-j, \lambda_2+j|\bullet\rangle$ (neutral, since we assume $\lambda_2\neq 0$). Then the images $\tilde{\delta}=(\delta;0)$ of the terms under $F$ are exactly the classes appearing in $\sigma_{n-2}\cdot\sigma_{\tilde{\gamma}}$. For any such $\tilde{\delta}$ we have $\gamma_1<n-2$ and $\delta_1>n-2$, so $\mathbb{D}=\mathbb{D}_1\neq \emptyset$ and $(1:n-1)\in \mathbb{D}_1$ is not killed. Then by Lemma~\ref{lemma:bisection} we have $N'(\gamma,\delta)=1$, so $\epsilon_{\tilde{\gamma},\tilde{\delta}}=\frac{1}{2}$ and $\tilde{\delta}$ has coefficient $1$.

\noindent {\bf Case 3:} ($|\olambda|> 2n-4$): If $\lambda_2>n-2$, then $\langle n-2,0|\circ\rangle\diamond\Pi(\olambda)=0$ by Lemma~\ref{lemma:firstrow}. Suppose $\lambda_2=n-2$. Then $\langle n-2,0|\circ\rangle^\uparrow\star\olambda=\frac{1}{2}\eta_{\olambda,\omu}\langle 2n-4,\lambda_1|\bullet\rangle$, assigned ${\rm ch}(\olambda)$ if $\lambda_1=n-2$. Let $\delta=f(\langle 2n-4,\lambda_1|\bullet\rangle) = (2n-3,\gamma_1-1)$. 
Since $\gamma_2=n-2$, $T'_2= R\setminus (1:n-1)$. Then $N'(\gamma,\delta)=0$ since $(1:n-1)$ is not a $\mathbb{D}$-box. Now, $g(\gamma,\delta)=n-2$ so $h(\tilde{\gamma},\tilde{\delta}) = n-2+{\tt type}(\gamma)$. Thus if $n$ is even, $\epsilon_{\gamma,\delta}=1$ if ${\tt type}(\gamma)=1$ and ${\tt type}(\delta)\in\{0,1\}$, and $\epsilon_{\gamma,\delta}=0$ otherwise. If $n$ is odd, $\epsilon_{\gamma,\delta}=1$ if ${\tt type}(\gamma)=2$ and ${\tt type}(\delta)\in\{0,2\}$, and $\epsilon_{\gamma,\delta}=0$ otherwise. This agrees with the coefficient $\frac{1}{2}\eta_{\olambda,\omu}$ of $\langle 2n-4,\lambda_1|\bullet\rangle$, and with the charge ${\rm ch}(\olambda)$ assigned if $\lambda_1=n-2$.

Now suppose $\lambda_2<n-2$. Then $\langle n-2,0|\circ\rangle\diamond\Pi(\olambda) = \sum_{0\le j\le M}\langle n-2+\lambda_1-j,\lambda_2+j|\bullet\rangle$. Here (i) has no effect, and since $n-2+|\lambda|\ge 3n-6$, every (legal) term $\okappa$ has $\kappa_2\ge n-2$, thus (ii) multiplies every term by $1$. There is an ambiguous term, namely $\langle 2n-4,n-2|\bullet\rangle$, if and only if $|\olambda|=2n-3$. Should it exist, it is disambiguated by (iii.3a). 
For the image $\delta$ of any term of $\langle n-2,0|\circ\rangle\diamond\Pi(\olambda)$, since $2n-2-\gamma_2\le \gamma_1+1$ we have $\mathbb{D}_1\setminus T'_1 = \emptyset$. Then if $\delta_2>n-2$, since $\gamma_2<n-2$ we have $(2:n-1)\in \mathbb{D}_2$ is not killed. So by Lemma~\ref{lemma:bisection}, $N'(\gamma,\delta)=1$. If $\delta=(2n-3,n-2)=f(\langle 2n-4,n-2|\bullet\rangle)$ then $\mathbb{D}_2=\emptyset$, so $N'(\gamma,\delta)=0$. Here $g(\gamma,\delta)=\gamma_2$, so $h(\tilde{\gamma},\tilde{\delta})=\gamma_2+{\tt type}(\delta)$. Thus $\epsilon_{\tilde{\gamma},\tilde{\delta}}=1$ if $\gamma_2$ is even and ${\tt type}(\delta)=1$ or if $\gamma_2$ is odd and ${\tt type}(\delta)=2$, while $\epsilon_{\gamma,\delta}=0$ otherwise. This agrees with the disambiguation (iii.3a) of $\langle 2n-4,n-2|\bullet\rangle$.

\noindent {\bf Case 4:} ($|\olambda|\le 2n-4$, $n-2+|\olambda|>2n-4$): 
There are three subcases.

\noindent {\bf Subcase 4a:} ($\lambda_1<n-2$): We compute $\langle n-2,0|\circ\rangle^\uparrow\star\olambda = $
\[\langle n-2+\lambda_1,\lambda_2-1|\bullet\rangle + 2\sum_{1\le j\le \lambda_1-\lambda_2} \langle n-2+\lambda_1-j,\lambda_2-1+j|\bullet\rangle + \langle n-2+\lambda_2-1,\lambda_1|\bullet\rangle \text{\ (neutral).}\] 
Then the images $\tilde{\delta}=(\delta;0)$ of the terms under $F$ are exactly the classes appearing in $\sigma_{n-2}\cdot\sigma_{\tilde{\gamma}}$. If $\tilde{\delta}=F(\langle n-2+\lambda_1,\lambda_2-1|\bullet\rangle)=\tilde{\gamma}^*$ then $N'(\gamma,\delta)=1$ by Lemmas~\ref{lemma:thresh} and \ref{lemma:specialPiericase}, hence $\epsilon_{\tilde{\gamma},\tilde{\delta}}=\frac{1}{2}$ and $\tilde{\delta}$ has coefficient $1$. For the image $\tilde{\delta}$ of a term in the summation, since $\delta_2<\gamma_1$ a component of $\mathbb{D}$ is bisected by Lemma~\ref{lemma:bisection}, thus $N'(\gamma,\delta)=2$ by Corollary~\ref{cor:bisection}. If $\tilde{\delta}=F(\langle n-2+\lambda_2-1,\lambda_1|\bullet\rangle)$ then $\delta_2=\gamma_1<n-2$, and since also $\delta_1>n-2$, we have $\mathbb{D}=\mathbb{D}_1$ and $(1:n-1)\in\mathbb{D}_1$ is not killed. Then $N'(\gamma,\delta)=1$ by Lemma~\ref{lemma:bisection}. 

\noindent {\bf Subcase 4b:} ($\lambda_1>n-2$): Let $M = \min\{\lambda_1-\lambda_2,n-2\}$. We compute $\langle n-2,0|\circ\rangle\diamond\Pi(\olambda) =$
\[\langle n-2+\lambda_1,\lambda_2-1|\bullet\rangle + 2\sum_{1\le j\le M}\langle n-2+\lambda_1-j,\lambda_2-1+j|\bullet\rangle + \langle n-2+\lambda_1-M-1,\lambda_2+M|\bullet\rangle.\] The first term is illegal.  
Now, (i) has no effect. Next, since $\lambda_2+M>n-2$, (ii) multiplies the last term by $1$, while for a term $\okappa$ of the summation, (ii) multiplies $\okappa$ by $\frac{1}{2}$ if $\kappa_2<n-2$ and by $1$ otherwise. Then (iii.3b) splits the ambiguous term of the summation.
For the image $\delta$ of any term $\okappa$, if $\delta_2=n-2$ we have both $(\delta;1)$ and $(\delta;2)$ appearing in $\sigma_{n-2}\cdot\sigma_{\tilde{\gamma}}$. This agrees with the splitting. Thus it remains to show that $N'(\gamma,\delta)=1$ for $\delta=f(\langle n-2+\lambda_1-M-1,\lambda_2+M|\bullet\rangle)$, while for all other $\delta$ we have $N'(\gamma,\delta)=1$ if $\delta_2\le n-2$ and $N'(\gamma,\delta)=2$ if $\delta_2>n-2$.

Consider the image $\delta$ of a term in the summation. If $\delta_2\le n-2$ then $\mathbb{D}=\mathbb{D}_1\neq \emptyset$, whence $N'(\gamma,\delta)=1$ by Lemma~\ref{lemma:bisection} and Lemma~\ref{lemma:threshcptsurvives}. If $\delta_2>n-2$, then since for any such $\delta$ we have $\delta_2<\gamma_1$, $\mathbb{D}=\mathbb{D}_1\cup\mathbb{D}_2$, where $\mathbb{D}_1,\mathbb{D}_2\neq\emptyset$ and $\mathbb{D}_1$ is not connected to $\mathbb{D}_2$. Then $N'(\gamma,\delta)=2$ follows from Lemma~\ref{lemma:bisection}, Lemma~\ref{lemma:threshcptsurvives} and the fact that (since $\gamma_2<n-2$) we have $(2:n-1)\in \mathbb{D}_2\setminus T'_1$. If $\delta=f(\langle n-2+\lambda_1-M-1,\lambda_2+M|\bullet\rangle)$ then either $\delta_2=\gamma_1$ or $\mathbb{D}_1=\emptyset$. In either case, $\mathbb{D}$ is a single connected component and $(2:n-1)\in\mathbb{D}_2$ is not killed. Then $N'(\gamma,\delta)=1$ follows from Lemma~\ref{lemma:bisection}.

\noindent {\bf Subcase 4c:} ($\lambda_1=n-2$): We compute $\langle n-2,0|\circ\rangle^\uparrow\star\olambda =$
\[\frac{1}{2}\eta_{\langle n-2,0|\circ\rangle^\uparrow,\olambda}\langle 2n-4,\lambda_2-1|\bullet\rangle + \sum_{1\le j\le n-2-\lambda_2}\langle 2n-4-j,\lambda_2-1+j|\bullet\rangle + \langle n-2+\lambda_2-1,n-2|\bullet\rangle^{{\rm ch}(\olambda)}.\]
If $\delta=f(\langle 2n-4,\lambda_2-1|\bullet\rangle)=\gamma^*$ then $N'(\gamma,\delta)=0$ by Lemmas~\ref{lemma:thresh} and \ref{lemma:specialPiericase}. Here $g(\gamma,\delta)=n-2$, so $h(\tilde{\gamma},\tilde{\delta})=n-2+{\tt type}(\gamma)$. Then $\epsilon_{\gamma,\delta} = 1$ if $n$ is even and ${\tt type}(\gamma)=1$, or if $n$ is odd and ${\tt type}(\gamma)=2$, while $\epsilon_{\gamma,\delta} = 0$ otherwise. This agrees with the coefficient $\frac{1}{2}\eta_{\langle n-2,0|\circ\rangle^\uparrow,\olambda}$ of $\langle 2n-4,\lambda_2-1|\bullet\rangle$. 

The $F$-image $\tilde{\delta}=(\delta;0)$ of a term in the summation has $\delta_2\le n-2$ and $\delta_1>n-2$, so $\mathbb{D}=\mathbb{D}_1\neq\emptyset$. Then by Lemma~\ref{lemma:bisection} and Lemma~\ref{lemma:threshcptsurvives}, $N'(\gamma,\delta)=1$. Therefore $\epsilon_{\tilde{\gamma},\tilde{\delta}}=\frac{1}{2}$, and the coefficient of $\tilde{\delta}$ is $1$.
For $\delta=f(\langle n-2+\lambda_2-1,n-2|\bullet\rangle)$, since $\delta_2<n-2$ and $\delta_1>n-2$ we have $\mathbb{D}=\mathbb{D}_1\neq\emptyset$. Then by Lemma~\ref{lemma:bisection} and Lemma~\ref{lemma:threshcptsurvives}, $N'(\gamma,\delta)=1$. Therefore $\epsilon_{\tilde{\gamma},\tilde{\delta}}=\frac{1}{2}$, and the coefficient of $\tilde{\delta}$ is $1$. We have only $\tilde{\delta}=(\delta;{\tt type}(\gamma))$ appearing in $\sigma_{n-2}\cdot\sigma_{\tilde{\gamma}}$, since ${\tt type}(\gamma)+{\tt type}(\delta)\neq 3$. This agrees with the charge assignment ${\rm ch}(\olambda)$.

\section*{Acknowledgements}

The author would like to thank Dave Anderson and Edward Richmond for helpful conversations. Thanks are also due to my advisor Alexander Yong whose suggestions and guidance greatly improved the exposition of this paper. The author was partially supported by NSF grants DMS 0901331 and DMS 1201595, and a grant from the UIUC Campus Research Board (13080).


\begin{thebibliography}{9999999999}

\bibitem[Be07]{Belkale} P.~Belkale, \emph{Geometric proof of a conjecture of Fulton}, Adv. Math. {\bf 216} (2007) no.~1, 346--357.
\bibitem[BeKu06]{Belkale.Kumar} P.~Belkale and S.~Kumar, \emph{Eigenvalue problem and a new product
in cohomology of flag varieties}, Invent.~Math., {\bf 166}(2006), 185--228.
\bibitem[BeKu10]{Belkale.Kumar10} \bysame, \emph{Eigencone, saturation and Horn problems for symplectic and odd orthogonal groups}, J. Algebraic Geom. {\bf 19} (2010), 199--242.
\bibitem[BeKuRe12]{Belkale.Kumar.Ressayre} P.~Belkale, S.~Kumar and N.~Ressayre, \emph{A generalization of Fulton's conjecture for arbitrary groups}, Mathematische Annalen {\bf 354} no.~2 (2012), 401--425.
\bibitem[BiBr03]{Billey.Braden} S.~Billey and T.~Braden, \emph{Lower bounds for Kazhdan-Lusztig polynomials from patterns}, Transform. Groups {\bf 8} (2003) no.~4, 321--332.
\bibitem[BuKrTa09]{BKT:Inventiones} A.~Buch, A.~Kresch and H.~Tamvakis, \emph{Quantum Pieri Rules for Isotropic Grassmannians}, Invent. Math. {\bf 178} (2009), 345--405.
\bibitem[Co13+]{Coskun.Orthogonal} I.~Coskun, \emph{Rigidity of Schubert classes in orthogonal Grassmannians}, Israel Journal of Mathematics, to appear.
\bibitem[CoVa09]{Coskun.Vakil} I.~Coskun and R.~Vakil, \emph{Geometric positivity in the cohomology of homogeneous spaces and generalized Schubert calculus}, in "Algebraic Geometry --- Seattle 2005" Part 1, 77--124, Proc. Sympos. Pure Math., 80, Amer. Math. Soc., Providence, RI, 2009.
\bibitem[KnPu11]{Knutson.Purbhoo} A.~Knutson and K.~Purbhoo,
\emph{Product and Puzzle Formulae for $GL_n$ Belkale-Kumar Coefficients}, Electr.~J.~Comb. 18(1): (2011).
\bibitem[KnTaWo04]{Knutson.Tao.Woodward} A.~Knutson, T.~Tao and C.~Woodward, \emph{The honeycomb model of $GL(n)$ tensor products II: Puzzles determine facets of the Littlewood-Richardson cone}, J. Amer. Math. Soc. {\bf 17} (2004) no.~1, 19--48.
\bibitem[PrRa96]{Prag} P.~Pragacz and J.~Ratajski,
\emph{A Pieri-type theorem for Lagrangian and odd orthogonal Grassmannians},
J. Reine Angew. Math. {\bf 476} (1996), 143--189.
\bibitem[PrRa03]{PragD} \bysame, \emph{A Pieri-type formula for even orthogonal Grassmannians}, Fund. Math. {\bf 178} (2003), 49-96.
\bibitem[Re10]{RessayreGIT} N.~Ressayre, \emph{Geometric Invariant Theory and Generalized Eigenvalue Problem}, Invent. Math. {\bf 180} no.~2 (2010), 389--441.
\bibitem[Re11]{RessayreFulton} \bysame, \emph{A short geometric proof of a conjecture of Fulton}, Enseign. Math. (2), Volume 57 (2011), 103--115.
\bibitem[Re13a+]{RessayreGIT2} \bysame, \emph{Geometric Invariant Theory and Generalized Eigenvalue Problem II}, Annales de l'Institut Fourier {\bf 61} no.~4 (2011), 1467--1491.
\bibitem[Re13b+]{RessayreEigencone} \bysame, \emph{A cohomology-free description of eigencones in types A, B, and C}, Inter. Math. Res. Notices {\bf 2012} no.~21 (2012), 4966--5005.
\bibitem[ReRi11]{Ressayre.Richmond} N.~Ressayre and E.~Richmond, \emph{Branching Schubert calculus and the Belkale-Kumar product on cohomology}, Proc. AMS, {\bf 139} (2011), 835--848.
\bibitem[Ri09]{Richmond} E.~Richmond, \emph{A partial Horn recursion in the cohomology of flag varieties}, J. Alg. Comb. (2009) {\bf 30} 1--17.
\bibitem[Sc77]{Schutzenberger} M.-P. Sch\"utzenberger, \emph{Combinatoire et repr\'esentation du groupe sym\'etrique}, (Actes Table Ronde CNRS, Univ. Louis-Pasteur Strasbourg, Strasbourg 1976), pp. 59--113. Lecture Notes in Math., Vol. 579, Springer, Berlin, 1977.
\bibitem[SeYo13]{Searles.Yong} D.~Searles and A.~Yong, \emph{Root-theoretic Young diagrams and Schubert calculus: planarity and the adjoint varieties}, preprint 2013. \textsf{arXiv:math/1308.6536}
\bibitem[Ta05]{Tam:qcig} H.~Tamvakis, \emph{Quantum cohomology of isotropic Grassmannians}, Geometric methods in algebra and number theory, Progr. Math. {\bf 235}, Birkhauser 2005, 311-338.
\bibitem[ThYo09]{Thomas.Yong:comin} H.~Thomas and A.~Yong, \emph{A combinatorial rule for
(co)minuscule Schubert calculus}, Adv. Math., {\bf 222}(2009), no.~2, 596--620.

\end{thebibliography}
\end{document}